\documentclass[12pt,twoside]{article}
\usepackage[all]{xy}
\usepackage {amssymb,latexsym,amsthm,amscd,amsmath,textcomp}
\usepackage[]{geometry}
\newcommand{\tensor}{\otimes}
\newtheorem{theorem}{Theorem}[section]
\newtheorem{lemma}{Lemma}[section]
\newtheorem{corollary}{Corollary}[section]
\newtheorem{proposition}{Proposition}[section]
\newtheorem*{definition}{Definition}
\newtheorem*{theorem*}{Theorem}
\theoremstyle{definition}

\numberwithin{equation}{subsection}

\newcommand{\ignore}[1]{}

\newcommand{\mynote}[1]{}

\title{\bf{On the Genus number of Algebraic groups}}
\author{ Anirban Bose\\ \small Indian Statistical Institute\\ \small 7, S.J.S. Sansanwal Marg, New Delhi-110016, India
\\ \small email: anirban.math@gmail.com}
\date{}
\begin{document}

\maketitle

\begin{abstract}
We compute the number of orbit types for simply connected simple algebraic groups over algebraically closed fields as well as for 
compact simply connected simple Lie groups. We also compute the number of orbit types for the adjoint action of these groups on their Lie 
algebras. We also prove that the genus number of a connected reductive algebraic group coincides with the genus number of its 
semisimple part.

\end{abstract}
\section{Introduction}
Let $G$ be a group acting on a set $M.$ Let $x\in M$ and $G_x$ denote the 
stabilizer of $x$ in $G$. Two elements $x,y\in M$ are said to have the same orbit type if the orbits of $x$ and $y$ are $G$-isomorphic,
 which is equivalent to saying
 $G_x$ is conjugate to $G_y$ in $G.$
 In the 1950s Mostow 
proved that for a compact Lie group acting on a compact manifold the number of orbit types is finite
\textbf{[M]}, which was initially conjectured by Montgomery (\textbf{[E]}, problem 45). The number of conjugacy
classes of centralizers of elements in a reductive algebraic group $G$ over an algebraically closed field, with char $G$ good, is finite
(\textbf{[St]}, Corollary 1 of Theorem 2, Chapter 3). 
Semisimple conjugacy classes for finite groups of Lie type have been studied by Fleischmann and Carter (see \textbf{[F],[C1]}). 
K. Gongopadhyay and R. Kulkarni have computed the number of conjugacy classes of centralizers in $I(\mathbb{H}^n)$
(the group of isometries of the hyperbolic $n-$space) \textbf{[GK]}. See \textbf{[K]}, 
where the author discusses a related notion of $z-$classes.  
Conjugacy 
classes  of centralizers in anisotropic groups of type $G_2$ over $\mathbb{R},$ have been explicitly calculated by 
A. Singh in \textbf{[S]}.
In this paper we compute the number of orbit types for the action of a 
compact simply connected simple Lie group on itself as well as for a simply connected 
simple algebraic group acting on its semisimple elements by conjugation. We also
compute the number 
of orbit types of the adjoint action of $G$ on its Lie algebra $\mathfrak{g}.$ 
We mainly do this for all classical groups and for $G_2$ and $F_4.$ Conjugacy classes of centralizers play an important role in the study
of characters of finite group of Lie type (see \textbf{[F], [C1]}). It seems natural that, an explicit knowledge of number of conjugacy 
classes of centralizers will be equally important. The remaining cases of exceptional groups will be handled in a sequel.
\section{Preliminaries}
The reader may refer to \textbf{[H2]} for basic results on algebraic groups and \textbf{[BD]} for the theory of compact connected 
Lie groups.

Let $G$ denote a compact simply connected Lie group or a simply connected algebraic group over an algebraically closed field 
and $T\subset G$ be a maximal torus of $G$. Let $W$ be the Weyl group of 
$G$ with respect to $T$, i.e. $W = N_G (T)/ T,$ where $N_G(T)$ denotes the normalizer of $T$ in $G.$ 
Conjugation induces an action of $W$ on $T$. For $x\in T $ let $W_x$ 
denote the stabilizer of $x$ in $W$ for this action i.e. $W_x = \{g\in W \colon gxg^{-1} =x \}$.
The cardinality of the set $\{[Z_G(x)]:x\in G, x\quad \rm{semisimple}\},$ where $Z_G(x)$ is the centralizer of $x$ in $G,$ 
is defined as the \textbf{semisimple genus number} of $G.$ 
Since we shall deal with only semisimple elements, we call this number simply as the \textbf{genus number } of $G.$ If $G$ is 
not simply connected, then the cardinality of the set $\{[Z_G(x)^\circ]:x\in G, x\quad \rm{semisimple}\},$ $Z_G(x)^\circ$ is the 
connected component of $Z_G(x)$ at the identity,
 is called the 
\textbf{connected genus number} of $G.$ 
  The following results are known:

\begin{proposition}{\label{conncent1}}(\textbf{[B]}, Theorem 3.4)
 Let $G$ be a simply connected compact Lie group and $\sigma\in Aut(G).$ Then the set $F$ of all fixed points of $\sigma$ in $G$
is connected. In particular, if $\sigma$ is the inner conjugation by an element $x\in G,$ then the centralizer $Z_G(x)$ is connected.
\end{proposition}

\begin{proposition}{\label{conncent2}}(\textbf{[H1]}, Theorem 2.11)
 If $G$ be a simply connected algebraic group over an algebraically closed field, the centralizer of any semisimple element of $G$ is
connected.
\end{proposition}
For a compact connected Lie group $G$ with maximal torus $T$ and Weyl group $W,$ define the following subsets with respect to a reflection
$s\in W:$
$T^{s}$ is the the subset of $T$ fixed by the action of $s\in W$ and $(T^{s})^\circ$ is the connected component at the identity of $T^{s}.$
Let $K(s)=\{x^2\in T | x\in N_G(T),\quad xT=s\in W\}$ and $\sigma(s)=(T^{s})^\circ\cup K(s).$ Then we have,  

\begin{proposition}{\label{centre}}(\textbf{[DW]}, Theorem 8.2)
 Suppose that $G$ is a compact connected Lie group with maximal torus $T$ and Weyl group $W.$ Then the centre of $G$ is equal to the 
intersection $\bigcap_{s}\sigma(s),$ where $s$ runs through the reflections in $W.$
\end{proposition}
 We have the following basic result:
\begin{theorem}{\label{genus}} For a simply connected compact Lie group $G$ with maximal torus $T$ and Weyl group $W,$
there exists a bijection 
$$
\{ [Z_G (x)]\colon x\in T\}\longrightarrow \{ [W_x]\colon x\in T\}
$$
 given by
$$
[Z_G (x)]\longmapsto[W_x]
$$
Here $[Z_G (x)]$ and $[W_x]$ respectively denote the conjugacy class of the centralizer of $x$ in $G$ and the conjugacy class of the 
stabilizer of $x$ in $W$.
\end{theorem} 

\begin{proof}First we show that the map is well-defined.
\newline
Let $x,y \in T$ such that $[Z_G (x)]=[Z_G (y)]$ i.e. there exists some $g\in G $ such that $gZ_G (x)g^{-1} =Z_G (y)$. Since $T$ is 
a maximal torus in $Z_G (x)$ containing $x$, $gTg^{-1}\subset Z_G (y)$ and also $T\subset Z_G (y).$ Hence there exists
 $g_1 \in Z_G (y)$ such that $g_1gTg^{-1}g_1^{-1} = T$. Let $g_1g = h \in G$. Then $[h]=hT \in W$ and
 $[h]W_x[h^{-1}] =W_y$
since, for $[h_1] \in W_x$ we have
\begin{align*}
 (hh_1h^{-1})y(hh_1^{-1}h^{-1})
=&(g_1gh_1g^{-1}g_1^{-1})y(g_1gh_1^{-1}g^{-1}g_1^{-1})\\
=&(g_1(gh_1g^{-1})g_1^{-1})y(g_1(gh_1^{-1}g^{-1})g_1^{-1})\\
=&y,
\end{align*}
since $h_1\in Z_G(x)$ and $gZ_G(x)g^{-1}=Z_G(y).$ Hence $gh_1g^{-1}\in Z_G(y).$ Also, $g_1\in Z_G(y).$
Therefore $[hh_1h^{-1}]\in W_y.$
Similarly we have the other inclusion. Thus the given map is well defined.

Surjectivity of the map is clear from the definition. Hence we only need to check injectivity.

Let $x,y\in T$ such that $W_x$ is conjugate to $W_y$, i.e. for some $ [h]\in W,$  $[h]W_x[h^{-1}] = W_y,$ i.e. 
$W_{hxh^{-1}}=W_y,$ where $h\in N_G(T)$ is a representative of $[h]\in W.$ We denote $hxh^{-1}\in T$ by $a.$ We intend to show that
$Z_G(a)=Z_G(y).$ Clearly for any element $x\in T,$ $W_x=N_{Z_G(x)}(T)/T.$ Therefore by Proposition \ref{centre}, 
$Z(Z_G(a))=\bigcap_{s\in W_a}\sigma(s) $ and $Z(Z_G(y))=\bigcap_{s\in W_y}\sigma(s).$ Since $W_a=W_y,$ we have 
$$Z(Z_G(a))=Z(Z_G(y))..........(\ast).$$
Observe that for any $x\in T,$ $Z_G(x)$ is the union of all maximal tori of $G$ containing $x.$ So let $T_1$ be any maximal torus
in $Z_G(a).$ Since $y\in Z(Z_G(a))$ by $(\ast),$ $y\in T_1, $ which implies $T_1\subset Z_G(y).$ Similarly any maximal torus of $Z_G(y)$
is contained in $Z_G(a).$ Therefore $Z_G(y)=Z_G(a)=Z_G(hxh^{-1})=hZ_G(x)h^{-1}.$ This shows that the map is injective. 
                                                                                   
\end{proof}
Next we prove an analogue of Theorem \ref{genus} for simply connected algebraic groups over an algebraically closed field. But before 
that, we note the following results:
\begin{proposition}{\label{centralizer1}}(\textbf{[C2]}, Theorems 3.5.3 and 3.5.4)
Let G be a connected reductive algebraic group, with maximal torus $T,$ Weyl group $W$ and root system $\Phi,$
then, for a semisimple element $x\in G,$ $Z_G(x)^\circ$ is  a reductive group and
\newline
$Z_G(x)^\circ=<T,U_{\alpha}, \alpha(x)=1>,$ where $\alpha\in \Phi$ and $U_\alpha$ is the root subgroup corresponding to $\alpha.$
\newline
The root system of $Z_G(x)^\circ$ is $\Phi_1=\{\alpha\in \Phi |\alpha(x)=1\}.$
\newline
The Weyl group of $Z_G(x)^\circ$ is $W_1=<w_{\alpha} |\alpha\in \Phi_1>,$ where $w_{\alpha}$ is the reflection at $\alpha.$
\end{proposition}
\begin{lemma}{\label{root}}
 Let $G$ be a simply connected algebraic group with maximal torus $T$ and Weyl group $W.$ If $w_\alpha$ be a reflection in $W,$ such
that $w_\alpha\in W_x,$ where
$x\in T$ and $\alpha\in \Phi,$ the root system of $G,$ then $\alpha(x)=1.$ 
\end{lemma}
\begin{proof}
 Let $(X(T),\Phi,Y(T),\Phi^\ast)$ be the root datum for $G.$ Since $G$ is simply connected, \\
$X(T)=Hom(\mathbb{Z}\Phi^\ast, \mathbb{Z})$
and $Y(T)=\mathbb{Z}\Phi^\ast.$ Therefore for a system of simple roots $\{\alpha_i\}$ of $G,$ there exists a basis
 $\{\lambda_j\}$ of $X(T)$ such
that $<\lambda_i, \alpha_j^\ast>=\delta_{ij},$ $\alpha_j^\ast$ being the coroot corresponding to $\alpha_j$ (
see \textbf{[SSt]}, Chapter 2, Section 2.) 

Now let $w_\alpha \in W$ be a reflection such that, $w_\alpha\in W_x,$ i.e. $w_\alpha(x)=x.$ 
There exists $s\in W$ such that $s(\alpha)$ is a simple root. Consider $\lambda \in X(T)$ such that $<\lambda, s(\alpha)^\ast>=1.$ 
Note that,
$$w_{s(\alpha)}(s(x))=sw_\alpha s^{-1}(s(x))=sw_\alpha(x)=s(x)..........(1).$$
Applying $\lambda$ to equation ($1$) we get,
\begin{align*}
 &\lambda (w_{s(\alpha)}(s(x)))=\lambda(s(x))\\
\Rightarrow&(w_{s(\alpha)}\lambda)(s(x))=\lambda(s(x))\\
\Rightarrow&(\lambda-<\lambda,s(\alpha)^\ast>s(\alpha))(s(x))=\lambda(s(x))\\
\Rightarrow&\lambda(s(x))s(\alpha)(s(x))^{-1}=\lambda(s(x))\\
\Rightarrow&s(\alpha)(s(x))=1\\
\Rightarrow&\alpha(s^{-1}(s(x))=1\\
\Rightarrow&\alpha(x)=1.
\end{align*}

\end{proof}

\begin{theorem}{\label{genus2}}
For simply connected algebraic group $G$ over an algebraically closed field, with maximal torus $T$ and Weyl group $W,$
there exists a bijection 
$$
\{ [Z_G (x)]\colon x\in T\}\longrightarrow \{ [W_x]\colon x\in T\}
$$
 given by
$$
[Z_G (x)]\longmapsto[W_x]
$$
Here $[Z_G (x)]$ and $[W_x]$ respectively denote the conjugacy class of the centralizer of $x$ in $G$ and the conjugacy class of the 
stabilizer of $x$ in $W$.
\end{theorem}
\begin{proof}
 The proof of well-definedness and surjectivity of the map is same as that in Theorem \ref{genus}. We prove that this map is injective.

Let $x,y\in T$ such that $W_x$ is conjugate to $W_y$, i.e. for some $ [h]\in W,$  $[h]W_x[h^{-1}] = W_y,$ i.e. 
$W_{hxh^{-1}}=W_y,$ where $h\in N_G(T)$ is a representative of $[h]\in W.$ We denote $hxh^{-1}\in T$ by $a.$ We intend to show that
$Z_G(a)=Z_G(y).$ To achieve this, we first show that $Z_G(a)$ and $Z_G(y)$ have the same roots. Let $\Phi_a$ and $\Phi_y$ respectively 
denote the root systems of $Z_G(a)$ and $Z_G(y)$ with respect to the common maximal torus $T.$ Since $G$ is simply connected,
 by Proposition \ref{conncent2}, both $Z_G(a)$ and $Z_G(y)$ are connected. Hence by Proposition \ref{centralizer1}, we have,
$\Phi_a=\{\alpha\in \Phi |\alpha(a)=1\}$ and $\Phi_y=\{\beta\in \Phi |\beta(y)=1\}.$ 

Let $\alpha\in \Phi_a.$ Hence $w_\alpha \in W_a=W_y. $ Therefore by Lemma \ref{root}, $\alpha(y)=1$ which implies $\alpha\in \Phi_y.$
This shows that $\Phi_a\subset \Phi_y.$ Similarly the other inclusion. Hence $\Phi_a=\Phi_y$ which implies $Z_G(a)=Z_G(y)$ by Proposition
\ref{centralizer1}. 
\end{proof}

\begin{corollary}
 Let $G$ be a compact simply connected Lie group (resp. a simply connected algebraic group over an algebraically closed field),
 $T\subset G$ a maximal torus. The genus number (resp. semisimple genus number) of $G$ equals the number of orbit types of the 
action of $W(G,T)$ on $T.$
\end{corollary}
\begin{proof}
By Theorem \ref{genus} and Theorem \ref{genus2}, 
the number of orbit types of elements belonging to a fixed maximal torus $T$ is equal to the number of orbit types 
of elements from $T$ in
the Weyl group. 
Any (semisimple) element $x\in G$ is contained in some maximal torus of $G$.
Let $y\in G$ be any other (semisimple) element and 
let $T^{\prime}$ be a maximal torus of $G$ such that $y\in T^{\prime}$. Now $T$ is conjugate to $ T^{\prime},$ i.e.
$\exists~ g\in G$ such that $gTg^{-1} = T^{\prime}$. Therefore $Z_G(y)$ is conjugate to $Z_G(x),$ where $x=g^{-1}yg\in T.$
Hence each (semisimple) element of $G$ is orbit equivalent to an element of $T.$ The result now follows.
\end{proof}
Next we want to investigate connected groups which are not necessarily simply connected. It turns out that the  connected genus number of
a connected semisimple group is equal to the genus number of its simply connected cover, which we shall see (Theorem 2.3).
We note the following two results, which are known:
\begin{proposition}(\textbf{[BD]}, Chapter 4,Theorem 2.9)
 Let $f;G\rightarrow H$ be a surjective homomorphism of compact Lie groups. If $T\subset G$ is a maximal torus, then $f(T)\subset
H$ is a maximal torus. Furthermore, ker$(f)\subset T$ iff ker$(f)\subset Z(G)$. In this case $f$ induces an isomorphism of Weyl groups.
\end{proposition}
A similar result holds for algebraic groups also, which we now quote (\textbf{[H2]}, Chapter 9, Proposition B),
\begin{proposition}
 Let $\phi:G\rightarrow G^\prime$ be an epimorphism of connected algebraic groups,
 with $T $ and $T^\prime=\phi(T)$ respective maximal tori. Then
$\phi$ induces a surjective map $WG\rightarrow WG^\prime,$ which is also injective in case Ker $\phi$ lies in all Borel subgroups of $G.$
Here, $WG$ and $WG^{\prime}$ denote the Weyl groups of $G$ and $G^{\prime}$ respectively.
\end{proposition}
Let $G$ be a compact connected semisimple Lie group or a
 connected semisimple algebraic group over an algebraically closed field. Let $\widetilde{G}$ be the simply connected cover of $G$ with 
the covering map,
$$\rho:\widetilde{G}\longrightarrow G.$$ Then, for a maximal torus
 $\widetilde{T}\subset \widetilde{G},$ $\rho(\widetilde{T})=T$ is a maximal torus in $G.$ Since $ker\rho$ is contained in all the maximal 
tori of $\widetilde{G},$ $\rho$ induces an isomorphism of $W\widetilde{G}$ and $WG$ by the above cited propositions.

Let $(X(T), \Phi, Y(T), \Phi^{\ast}) $ be the root datum of $G.$ Let  $V:= (Y(T)\tensor \mathbb{R})$ and 
$\overline{Y(T)}:= \{v\in V: \alpha(v)\in \mathbb{Z},\quad \forall \alpha\in \Phi\}.$
  We associate a finite group $C:= \overline{Y(T)}/\mathbb{Z}\Phi^{\ast}$ with the isogeny class of $G.$ Then $C$ is a
 finite abelian group. Let $C^{\prime}(G):=Y(T)/\mathbb{Z}\Phi^{\ast}\subset C.$ It can be shown that any subgroup of $C$ is of the form 
$C^{\prime}(H), $
for some group $H$ belonging to the isogeny class of $G.$  (see \textbf{[T]}, Section 1.5) 

We first make the following observation:
\begin{lemma}
 Let $G$ be a compact connected semisimple Lie group or a connected semisimple algebraic group over an algebraically closed field $K$
and $\widetilde{G}$
be its simply connected cover. Let $\rho: \widetilde{G}\rightarrow G$ be the covering map. Assume that, $char(K)$ does not divide the order
of $ C(G).$
 Then 
$\rho(Z_{\widetilde{G}}(\tilde{x}))=Z_G(x)^\circ,$ where $\tilde{x}\in \widetilde{T},$ a fixed maximal torus in $\widetilde{G}$ and 
$x=\rho(\tilde{x}).$
\end{lemma}
\begin{proof}
For an algebraic group or a Lie group $G,$ let us denote the corresponding Lie algebra by $\mathbf{L}(G).$
 Since $char(K)$ does not divide the order of $C^{\prime}(G),$ $\rho$ is a separable morphism. Hence, the differential
$d\rho:\mathbf{L}(\widetilde{G})\rightarrow \mathbf{L}(G),$ is an isomorphism of Lie algebras.
Since $Z_{\widetilde{G}}(\tilde{x})$ is connected, $\rho(Z_{\widetilde{G}}(\tilde{x}))\subset Z_G(x)^\circ.$
 If we show that the dimensions
are equal, we would be through. For this, we look at the corresponding Lie algebras. 
 Now since $Ad_{x}v=v$ for all 
$v\in\mathbf{L}(Z_G(x)^\circ), $ $d\rho Ad_{\tilde{x}}d\rho^{-1}v=Ad_{x}v=v.$  Therefore, 
for every $v\in \mathbf{L}(Z_G(x)^\circ),$ $Ad_{\tilde{x}}d\rho^{-1}v=d\rho^{-1}v.$ 
Hence $d\rho^{-1}(\mathbf{L}(Z_G(x)^\circ))\subset \mathbf{L}(\rho(Z_{\widetilde{G}}(\tilde{x}))) .$ Since $d\rho$ is an isomorphism,
we have $dim(\mathbf{L}(Z_G(x)^\circ))\leq dim (\mathbf{L}(\rho(Z_{\widetilde{G}}(\tilde{x})))).$
Therefore $dim( Z_G(x)^\circ)\leq dim( \rho(Z_{\widetilde{G}}(\tilde{x}))). $ Hence the equality.

\end{proof}
\noindent\textbf{Remark:} Note that, the covering map $\rho:SL_2(K)\longrightarrow PSL_2(K),$ is not separable if $char(K)=2,$ 
since $C^{\prime}(PSL_2(K))= \mathbb{Z}_2.$ Hence in this case, $d\rho$ is not an isomorphism.

\begin{theorem}
 Let $G$ be a compact connected semisimple Lie group or a connected semisimple algebraic group over an algebraically closed field $k.$
Let $\widetilde{G}$ be the simply connected cover of $G$ with the covering map $\rho.$ Fix  a maximal torus
$\widetilde{T}$ in $\widetilde{G}.$ Then the map,
$$\{[Z_{\widetilde{G}}(\tilde{t})]:\tilde{t}\in\widetilde{T}\}\rightarrow\{[Z_G(x)^\circ]:x\in T\}$$
defined by,
$$[Z_{\widetilde{G}}(\tilde{t})]\mapsto [Z_G(\rho(\tilde{t}))^\circ],$$
is a bijection. Here $T=\rho(\widetilde{T})\subset G$ is a maximal torus.
\end{theorem}
\begin{proof}
If $\tilde{g}\in \widetilde{G},$ then we shall denote $\rho({\tilde{g}})$ by $g.$

  We first show that the map is well-defined. So let, $[Z_{\widetilde{G}}(\tilde{t})]=[Z_{\widetilde{G}}(\tilde{t_1})]$ with 
$\tilde{t},\tilde{t_1}\in \widetilde{T}.$ Therefore there exists $\tilde{g}\in \widetilde{G}$ such that,
$Z_{\widetilde{G}}(\tilde{t})=\tilde{g}Z_{\widetilde{G}}(\tilde{t_1})\tilde{g}^{-1}=Z_{\widetilde{G}}(\tilde{g}\tilde{t_1}\tilde{g}^{-1}).$
Take $a\in Z_G(t)^\circ,$ where $\rho(\tilde{t})=t.$ Consider any lift $\tilde{a}\in Z_{\widetilde{G}}(\tilde{t})$ of $a$
(such a lift exists by Lemma 2.2). Therefore,
$\tilde{a}\tilde{g}\tilde{t_1}\tilde{g}^{-1}\tilde{a}^{-1}=\tilde{g}\tilde{t_1}\tilde{g}^{-1}.$ Applying $\rho $ on both sides we get,
$agt_1g^{-1}a^{-1}=gt_1g^{-1}.$ Thus, $Z_G(t)^\circ \subset Z_G(gt_1g^{-1})^\circ.$ Similarly $Z_G(gt_1g^{-1})^\circ \subset Z_G(t)^\circ.$

That the map is onto is clear from the definition.

To prove that the map is injective, let $Z_G(t_1)^\circ=gZ_G(t_2)^\circ g^{-1}=Z_G(gt_2g^{-1})^\circ$ for some $g\in G.$ If $\tilde{a}
\in Z_{\widetilde{G}}(\tilde{t_1}),$ the $a=\rho(\tilde{a})\in Z_G(t_1)^\circ=Z_G(gt_2g^{-1})^\circ.$ Therefore,
$agt_2g^{-1}a^{-1}=gt_2g^{-1}.$ If we show that $\tilde{a}\in Z_{\widetilde{G}}(\tilde{g}\tilde{t_2}\tilde{g}^{-1})$ then we are through.
So let $\tilde{a_1} $ be any  lift of $a$ in $Z_{\widetilde{G}}(\tilde{g}\tilde{t_2}\tilde{g}^{-1}).$ Then, 
$\rho(\tilde{a}\tilde{a_1}^{-1})=1\Rightarrow \tilde{a}\tilde{a_1}^{-1}\in Ker \rho\subset Z(\widetilde{G}).$ Therefore,
$\tilde{a}\tilde{a_1}^{-1}\tilde{g}\tilde{t_2}\tilde{g}^{-1}\tilde{a_1}\tilde{a}^{-1}=\tilde{g}\tilde{t_2}\tilde{g}^{-1}
\Rightarrow \tilde{a}\tilde{g}\tilde{t_2}\tilde{g}^{-1}\tilde{a}^{-1}=\tilde{g}\tilde{t_2}\tilde{g}^{-1}.$ Hence, $\tilde{a}\in 
Z_{\widetilde{G}}(\tilde{g}\tilde{t_2}\tilde{g}^{-1}),$ which shows that
 $Z_{\widetilde{G}}(\tilde{t_1})\subset Z_{\widetilde{G}}(\tilde{g}\tilde{t_2}\tilde{g}^{-1}).$ Similarly 
the other inclusion follows. This completes the proof.
\end{proof}
\textbf{Remark}: It is important to note that if the group is not simply connected, then the number of classes of centralizers 
might be larger than
the number of isotropy classes of the Weyl group. For example if we consider the group $PSL_2(K)$( $char(K)\neq 2$), the number of isotropy subgroups 
in the Weyl
group $S_2$ is $2$ but the number of conjugacy classes of centralizers is $3.$ However, by Theorem 2.3, the connected genus number of 
$PSL_2(K)$ is $2$ which is equal to the genus number of its simply connected cover $SL_2(K).$

We have the following result on reductive algebraic groups:
\begin{theorem}
 Let $G$ be a connected reductive algebraic group over an algebraically closed field. Let $G^\prime$ be the commmutator subgroup of $G.$
Then the connected genus number of $G$ is equal to the connected genus number of $G^\prime.$
\end{theorem}
\begin{proof}
 Since $G$ is reductive, we have $G=G^\prime.Z(G)^\circ,$ where $Z(G)^\circ$ is the connected component of the centre of $G.$ 
For any $g\in G,$ we shall write $g=g^\prime s_g,$ with $g^\prime \in G^\prime$ and $s_g\in Z(G)^\circ.$
Observe that
for any $g^\prime \in G^\prime$ and $s\in Z(G)^\circ,$ $Z_G(g^\prime s)=Z_G(g^\prime).......(\ast).$ 

Define a map:
$$\{[Z_G(x)^\circ]:x \quad  semisimple \}\rightarrow \{[Z_{G^\prime}(x^\prime)^\circ]: x^\prime \quad  semisimple\}$$
by, $[Z_G(x)^\circ]\mapsto [Z_{G^\prime}(x^\prime)^\circ],$ where $x=x^\prime s_x,$ $x^\prime \in G^\prime$ and $s_x\in Z(G)^\circ.$ 
We prove that this map is a bijection.

To show that the above map is well defined, assume that $Z_G(x)^\circ= Z_G(gyg^{-1})^\circ,$ for some $g\in G.$ 
Then by $(\ast),$ $Z_{G^\prime}(x^\prime)^\circ \subset Z_G(x^\prime)^\circ= Z_G(x)^\circ=Z_G(gyg^{-1})^\circ=Z_G(gy^\prime g^{-1})^\circ.$
Hence $Z_{G^\prime}(x^\prime)^\circ \subset Z_{G^\prime}(gy^\prime g^{-1})^\circ.$ 
Similarly $Z_{G^\prime}(gy^\prime g^{-1})^\circ\subset Z_{G^\prime}(x^\prime)^\circ , $ which shows that the above map is well defined.

It is clear from the definition that the map is onto.

We now prove the injectivity. So assume that, $Z_{G^\prime}(x^\prime)^\circ=Z_{G^\prime}(g^\prime y^\prime g^{\prime -1})^\circ,$ 
for some $g^\prime \in G^\prime.$ Let $a\in Z_G(x^\prime)^\circ,$ where $a=a^\prime s_a.$ Then $a^\prime \in Z_G(x^\prime)$ as $s_a$ is 
central. Also note that $s_a\in Z(G)^\circ\subset Z_G(x^\prime)^\circ.$ Therefore, $a^\prime= as^{\prime -1}\in Z_G(x^\prime)^\circ.$ 
In particular, $a^\prime \in Z_{G^\prime}(x^\prime).$ 

We claim that $a^\prime \in Z_{G^\prime}(x^\prime)^\circ.$ If $a^\prime $ is
 unipotent, then $a^\prime \in Z_{G^\prime}(x^\prime)^\circ,$ since $G^\prime$ is a connected semisimple 
group(see \textbf{[H1]}, Chapter 1, Section 12). So let  $a^\prime$ be semisimple. Choose a maximal torus $T\in Z_G(x^\prime)^\circ$ 
such that $a^\prime \in T.$ Let $T=T^\prime. Z(G)^\circ,$ where $T^\prime$ is a maximal torus in $G.$ Therefore,
 $T^\prime\subset Z_{G^\prime}(x^\prime)^\circ.$  Write $a^\prime=a_1 b$ with $a_1\in T^\prime$ and $b\in Z(G)^\circ.$ 
Since both $a_1$ and $b$ are in $Z_{G^\prime}(x^\prime)^\circ,$ so is $a^\prime.$ Hence the claim. Therefore, by assumption, 
$a^\prime \in Z_{G^\prime}(x^\prime)^\circ=Z_{G^\prime}(g^\prime y^\prime g^{\prime -1})^\circ \subset 
Z_G(g^\prime y^\prime g^{\prime -1})^\circ.$ Since $a_s\in Z(G)^\circ,$ $a=a^\prime a_s\in Z_G(g^\prime y^\prime g^{\prime -1})^\circ. $
 Thus we have shown that, $Z_G(x^\prime)^\circ\subset Z_G(g^\prime y^\prime g^{\prime -1})^\circ.$ Similarly the other inclusion follows.
Hence the map is injective. 
\end{proof}
\textbf{Remark:} By Theorem 2.4, the genus number of $GL_n(k)$ is equal to the genus number of $SL_n(k).$

\subsection*{Disconnected centralizers}
In general, for a connected semisimple group we can derive a necessary and sufficient condition for connectedness of centralizers of
semisimple elements.
Let $G$ be a connected semisimple algebraic group, with the simply connected cover $\widetilde{G}$ and 
 $\rho: \widetilde{G}\longrightarrow G$ 
be the covering map. Let $T\subset G$ be a fixed 
maximal torus. Consider $t\in T$ and let $\rho^{-1}(t)= \{\tilde{t_1},...,\tilde{t_l}\} \subset \widetilde{G}.$ 
Then we have the following:

\begin{theorem}
 Fix a lift $\tilde{t_1}\in \widetilde{G}$ of $t\in T$. Then $Z_G(t)$ is disconnected if and only if there exists $\tilde{g}\in 
\widetilde{G}$ such that, $\tilde{g}\tilde{t_1}\tilde{g}^{-1}= \tilde{t_i},$ for some $i\neq 1.$ 
\end{theorem}
\begin{proof}
Let $Z_G(t)$ be disconnected. Therefore, there exists $g\in Z_G(t)\setminus Z_G(t)^\circ.$ Let $\tilde{g}\in \widetilde{G}$ be a lift
of $g.$ Observe that $\rho(\tilde{g}\tilde{t_1}\tilde{g}^{-1})= gtg^{-1}=t.$ So, $\tilde{g}\tilde{t_1}\tilde{g}^{-1}\in \rho^{-1}(t).$
Also note that $\tilde{g}\tilde{t_1}\tilde{g}^{-1}\neq \tilde{t_1}.$ For else, $\tilde{g}\in Z_{\widetilde{G}}(\tilde{t_1}),$ which 
implies $\rho(\tilde{g})\in \rho(Z_{\widetilde{G}}(\tilde{t_1}))$ $\Rightarrow$ $g\in Z_G(t)^\circ$ (since $Z_{\widetilde{G}}(\tilde{t_1})$
is connected). Hence $\tilde{g}\tilde{t_1}\tilde{g^{-1}}= \tilde{t_i}$ for some $i\neq 1.$

Conversely, let there exist $\tilde{g}\in \widetilde{G},$ such that, $\tilde{g}\tilde{t_1}\tilde{g}^{-1}=\tilde{t_i},$ for some
$i\neq 1.$ Therefore $g=\rho(\tilde{g})\in Z_G(t).$ Define $S_j=\{x\in Z_G(t)| \tilde{x}\tilde{t_1}\tilde{x}^{-1}=\tilde{t_j}\},$ where 
$\rho(\tilde{x})=x.$ Then clearly, $Z_G(t)= \bigcup_{j=1}^n S_j.$ Note that, $S_1= \rho(Z_{\widetilde{G}}(\tilde{t_1}))=Z_G(t)^\circ$ and
by hypothesis, $S_i$ is non empty. Hence $Z_G(t)$ is not connected.
\end{proof}
In what follows, we shall compute the genus number of all the compact simply connected simple Lie groups
and simply connected simple algebraic groups of Classical type and of types $G_2$ and $F_4.$ 
\section{$A_n$}

In this section, we compute the genus number for the compact Lie group $SU(n+1)$ and the semisimple genus number of the algebraic
group $SL(n+1)$ over an algebraically closed field.
We fix a maximal torus $T$ of $SU(n+1)$ consisting of all matrices of the form 

$$\begin{bmatrix}
 z_1&{}&{}\\
 {}&\ddots&{}\\
{}&{}&z_{n+1}
\end{bmatrix},$$
where $z_i\in S^1$ and $z_1...z_{n+1} = 1$.
If we write $z_l =exp(2\pi i\gamma _l)$, then the above matrix can be represented by
the $(n+1)$-tuple $(\gamma_1,\gamma_2...,\gamma_{n+1}),$ where $\gamma_i\in \mathbb{R}/\mathbb{Z}$. The Weyl group of $SU(n+1)$ is $S_{n+1}$ 
and it acts on the diagonal maximal torus in the following way: let $\alpha \in S_{n+1}$ and $\gamma =(\gamma_1,\gamma_2,...,\gamma_{n+1})
\in T,$ then
$
\alpha^{-1}(\gamma_1,\gamma_2,...,\gamma_{n+1})$ =$ (\gamma_{\alpha(1)},\gamma_{\alpha(2)},..,$ $\gamma_{\alpha(n+1)}).$

We wish to compute the number of conjugacy classes of isotropy subgroups of $S_{n+1}$ with
respect to its action on  $T.$

Let $\gamma \in T.$ By the action of a suitable element of $S_{n+1}$ we can assume $\gamma $ to be such
that, $\gamma_1=\gamma_2=...=\gamma_{k_1};~ \gamma_{k_1+1}=...=\gamma_{k_1+k_2}~; ...;~ \gamma_{k_1+...+k_{l-1}+1}=...=\gamma_{k_1+...k_l}$ 
and 
$k_1+k_2+...k_l =n$, with $\gamma_1\neq \gamma_{k_1+1}\neq...\neq \gamma_{k_1+...+k_{l-1}+1}.$ 
Hence, for this $\gamma,$ the isotropy subgroup in $S_{n+1}$ is $S_{k_1}\times S_{k_2}\times ...\times S_{k_l}
\subset S_{n+1},$ where 
$S_{k_i}=\{\rho\in S_{n+1}|~\rho(j)=j~for~j=1,...,(k_1+...+k_{i-1}),(k_1+...+k_i+1),...,n+1\}.$
 Note that $S_{k_i} \cap S_{k_j}=\{1\}$ for $i\neq j$ and $S_{k_i}S_{k_j}= S_{k_j}S_{k_i}.$ So,
 $S_{k_i}S_{k_j}$ is a subgroup of $S_{n}$ and hence by induction $S_{k_1}...S_{k_n}$ is a subgroup of $S_{n}$.

More precisely, any element 
$\rho\in W_{\gamma},$ necessarily has a cycle decomposition of the type $(k_1,...,k_l)$, i.e. $\rho\in S_{k_1}S_{k_2}...S_{k_l}$
and conversely any element of $S_{k_1}\times S_{k_2}\times ...\times S_{k_l}$ is clearly a stabilizer of $\gamma$.
In other words, we have the following 
isomorphism :
$$
W_{\gamma}\longrightarrow S_{k_1}...S_{k_l}
$$
$$
\rho\longmapsto (\rho\mid_{k_1}.\rho\mid_{k_2}...\rho\mid_{k_l}),
$$ 
where $\rho\mid_{k_i}$ denotes the restriction of $\rho$ on to the $k_i$ many entries of $\gamma$, which are equal modulo $\mathbb{Z}.$

Let $(n_1,...,n_l)$ and $(m_1,...,m_k)$ be two ordered partitions of $n+1$ and suppose they correspond to elements 
$\gamma_1,\gamma_2\in T$
respectively. If $l=k$ and $n_i=m_i$ for all $1\leq i\leq l,$ clearly $W_{\gamma_1}=W_{\gamma_2}.$ Now suppose that
 the two partitions are different.
Then  $n_i\neq m_i$ for some $i.$ We observe that any element in
$ W_{\gamma_1}$ has a cycle type $(n_1,...,n_l)$ and any element in $W_{\gamma_2}$ has cycle type $(m_1,...,m_k)$ and since 
conjugation in $S_n$ must preserves cycle types, $W_{\gamma_1}$ is not conjugate to $ W_{\gamma_2}.$
 
Thus the number of conjugacy classes of isotropy subgroup is precisely $p(n+1)$, i.e. the number of partitions of $n+1$.

For $SL(n+1)$ over an algebraically closed field $k,$ the semisimple genus
number is similarly obtained by computing the number of isotropy subgroups of the Weyl group (up to conjugacy) with respect to
its action on a maximal torus. In this situation again we consider the diagonal maximal torus $T\subset SL(n+1)$, i.e the subgroup
of matrices of the form $diag(a_1,...,a_{n+1})$ such that $a_1...a_{n+1}=1,~a_i\in k.$ Following a similar argument as in the case of 
$SU(n+1),$ we see that the number of conjugacy classes of isotropy subgroups of Weyl group is $p(n+1).$

 We record this as :
\begin{theorem}
 The genus number of a compact simply connected Lie group 
or a simply connected algebraic group over an algebraically closed field, of type $A_n$
 is $p(n+1).$
\end{theorem}
\section{$B_n$}
We consider the simply connected group $Spin(2n+1) $ and a maximal
 torus $$T=\{\prod_{i=1}^n (cos t_i-e_{2i-1}e_{2i}sin t_i):0\leq t_i\leq2\pi\}.$$ To simplify notations let us denote a typical element
of $T$ by $t=(t_1,...,t_n),$ with $0\leq t_i\leq 2\pi.$

For  a description of the Weyl group of $Spin(2n+1),$ we fix the following notation:
 
$$
t_{-i} = -t_i,\quad \rm for\quad i=1,...,n.
$$
The Weyl group of $Spin(2n+1)$ is $W=(\mathbb{Z}/2)^n\rtimes S_n,$ where $S_n$ acts on $(\mathbb{Z}/2)^n$ by permuting the coordinates.
The group $W$ can be identified with the group of permutations $\phi$ of the set $\{-n,...,-1,1,...,n\},$ which satisfy 
$\phi(-i)= -\phi(i).$
$W$  acts on the fixed maximal torus $T$ of
$Spin(2n+1)$ in the following way:
$$
\phi(t_1,...,t_n) = (t_{\phi^{-1}(1)},...,t_{\phi^{-1}(n)}),
$$
where $\phi\in W$ and $(t_1,...,t_n)\in T$. 

\noindent\textbf{A useful interpretation:} The action of $W$ on the maximal
torus of $Spin(2n+1)$ can be described in the following way:\\
An element $\phi\in G(n)$ acts on a toral element $t\in T$ by permuting the parameters and changing the sign of some of them.
 If $\phi= (\alpha,\beta),$ with $\alpha\in (\mathbb{Z}/2)^n$ and $\beta\in S_n$, then $\beta$ permutes the parameters of $t$ and 
$\alpha$ changes the signs of the parameters.

In order to compute the number of conjugacy classes of isotropy subgroups of $W,$
 we start with an element $t=(t_1,...,t_n)\in T$ and find the isotropy subgroup $W_t.$

Let $n=n_1+...+n_k,$
where, $t_i=0$ or $\pi,$ for $i=1,...,n_1^\prime,$ $t_i=\pi/2$ or $3\pi/2,$ for $i=n_1^\prime+1,...,n_1, $ 
and $t_i\neq 0,\pi,\pi/2,3\pi/2$ for $i\geq n_1+1.$ The remaining integers $n_2,...,n_k$ denote the number of parameters which are equal.

Note that, for $i=1,...,n_1^\prime,$ a non- trivial $(\mathbb{Z}/2)^n$ action on $t_i$ fixes the factor $(cos t_i-e_{2i-1}e_{2i}sin t_i),$
which is $1$ or $-1$ according as $t_i=0$ or $\pi.$ However, for $i=n_1^\prime+1,...,n_1,$ a non-trivial 
$(\mathbb{Z}/2^n)$ action on $t_i$ 
inverts the factor $(cos t_i-e_{2i-1}e_{2i}sin t_i),$ which is $ e_{2i-1}e_{2i}$ or $-e_{2i-1}e_{2i},$ 
according as $t_i=\pi/2$ or $3\pi/2.$ For the 
rest of the parameters, only the $S_n$ part of the Weyl group contributes to the isotropy. Therefore
the isotropy subgroup for such an element of $T$ is
$$
((\mathbb{Z}/2)^{n_1^\prime}\rtimes S_{n_1^\prime})\times ((\mathbb{Z}/2)^{n_1-n_1^\prime-1}\rtimes 
S_{n_1-n_1^\prime})\times S_{n_2}\times...\times S_{n_k}, 
$$

Therefore for each choice of $n_1$ we have $(n_1+1)p(n-n_1)$ many isotropy subgroups (non-conjugate).\\
Hence the total number of conjugacy classes of isotropy subgroups of 
$W$ for $SO(2n+1)$ is $$\sum_{i=0}^n (i+1)p(n-i).$$

When we consider $Spin(2n+1)$ over an algebraically closed field $k,$ we take a maximal torus
$T=\{\prod_{i=1}^n(t_i^{-1}+(t_i-t_i^{-1}e_{2i-1}e_{2i}),  t_i\in k^\ast\}.$ We can calculate the number
of conjugacy classes of isotropy subgroups of the Weyl group using similar arguments.

We record this discussion as:
\begin{theorem}
 The genus number of a compact simply connected Lie group or a simply connected  algebraic group over an 
algebraically closed field, of type $B_n$ is  $\sum_{i=0}^n (i+1)p(n-i).$
\end{theorem}
\begin{corollary} 
The connected genus number of $SO(2n+1)$ is equal to the genus number of $Spin(2n+1).$
 
\end{corollary}
\begin{proof}
 Follows from Theorem 2.3.
\end{proof}

\section{$C_n$}
Let k be an algebraically closed field. The symplectic group over $k$ of rank $n,$ is defined as
$Sp(n,k):=\{A\in GL_{2n}(k): A^tJA=J\},$ where $J=\begin{bmatrix}
                                                 0&-I\\
                                                 I&0
                                                \end{bmatrix},
$ $I$ being the identity matrix in $GL_n(k).$ $Sp(n,k)$ is the simply connected algebraic group of type $C_n.$

When $k=\mathbb{C},$ the field of complex numbers, $Sp(n,\mathbb{C})$ is the complex symplectic group of rank $n.$ The compact simply 
connected Lie group of type $C_n,$ denoted by $Sp(n)$ is defined as follows: 
let $U(n)$ denote the group of $n\times n$ unitary matrices. Define
 $Sp(n):= \{A\in U(2n):A^tJA= J\},$
where $J=$ $\begin{bmatrix}
           0&-I\\
           I&0
          \end{bmatrix}$, $I$ is the identity matrix in $GL_n(\mathbb{C}).$ Therefore, $Sp(n)= Sp(2n,\mathbb{C})\cap U(2n).$
 We have the inclusion $U(n)\longrightarrow Sp(n),$ given by $A\mapsto \begin{bmatrix}
                                                                              A&0\\
                                                                              0&\overline{A}
                                                                            \end{bmatrix}$.\\
Consider the maximal torus $T(n)=\left\{\begin{bmatrix}
                                z_1&{}&{}\\
                                 {}&\ddots&{}\\
                                  {}&{}&z_n
                               \end{bmatrix}\in U(n): z_i\in S^1\right\}\subset U(n).$
Then the image of this maximal torus in $Sp(n)$ under the above inclusion gives a maximal torus $T\subset Sp(n)$,
 a typical element of which 
is of the form,
$$
t= \begin{bmatrix}
      z_1&{}&{}&{}\\
      {}&z_2&{}&{}\\
      {}&{}&\ddots&{}\\
      {}&{}&{}&z_n&&{}&{}\\
      {}&{}&{}&{}&\overline{z_1}&{}&{}&{}\\
      {}&{}&{}&{}&{}&{}\overline{z_2}&{}&{}\\
      {}&{}&{}&{}&{}&{}&{}&\ddots&{}\\
      {}&{}&{}&{}&{}&{}&{}&{}&\overline{z_n}
\end{bmatrix}.
$$
Let $z_k= exp(2\pi it_k).$ Then we can represent each $t\in T$ by an $n$-tuple $(t_1,...,t_n),$ where $t_k\in \mathbb{R/Z}.$

The Weyl group of $Sp(n)$ is $W=(\mathbb{Z}/2)^n\rtimes S_n$,
where $S_n$ acts on $(\mathbb{Z}/2)^n$ by permuting the coordinates, as noted in Section 4.
The action of $W$ on $T$ is given by, $\phi(t_1,...,t_n) = (t_{\phi^{-1}(1)},...,t_{\phi^{-1}(n)}),$ 
where $\phi\in W$ and $(t_1,...,t_n)\in T.$ We follow the same convention: $t_{-i}=-t_i,$ for $i=1,...,n$ (see Section 4).

To compute the isotropy subgroup of $t\in T$ in $W,$ first note that, if $t_i=0$ or $1/2,$ a non-trivial
$(\mathbb{Z}/2)^n$ action fixes $t_i.$ Therefore, 
we can assume without loss of generality that, $t_i\neq-t_j$ unless 
$t_i=t_j=0,1/2.$ For, if there exist $t_i=-t_j$ for some $i,j$ with $t_i,t_j\neq 0, 1/2$ then we can change the sign of $t_j$ by suitable
element from $(\mathbb{Z}/2)^n.$

Let $n=n_1+...+n_k$ be a partition of $n$ with $n_1$ being the total number of $0^,$s and $1/2^,$s and $n_2,...,n_k$ are 
the sizes of the blocks of parameters $t_i$ which are equal. The isotropy subgroup for
this particular $t$ is
$$((\mathbb{Z}/2)^i\rtimes S_i)\times ((\mathbb{Z}/2)^{n_1-i}\rtimes S_{n_1-i})\times S_{n_2}\times...\times S_{n_k},$$
where $i$ and $n_1-i$ respectively denote the number of $0^,$s and $1/2^,$s in $t$. Therefore for this partition of $n,$ we have $([n_1/2]+1)p(n-n_1)$ many distinct isotropy 
subgroups (by varying the number of $0^,$s). Hence the total number of conjugacy classes of isotropy subgroups is
$$\sum_{i=0}^n ([i/2]+1)p(n-i).$$\\

Over an algebraically closed field $k$,
the diagonal maximal torus of $Sp(n)$ can again be parametrized by $n$ coordinates $(a_1,...,a_n)$ $a_i\in k^*$. The 
calculation for genus number follows exactly as above.Thus we have the following:
\begin{theorem}
 The genus number of a compact simply connected Lie group 
or a simply connected algebraic group over an algebraically closed field, of type $C_n$
 is $\sum_{i=0}^n ([i/2]+1)p(n-i)$.
\end{theorem}
\section{ $D_n$}
Here, as in the case of $Spin(2n+1),$ we work with the maximal torus 
$T=\{\prod_{i=1}^n (cos t_i-e_{2i-1}e_{2i}sin t_i):0\leq t_i\leq2\pi)\}.$ The 
Weyl group is $W= (\mathbb{Z}/2)^{n-1}\rtimes S_n,$ the subgroup of even permutations in the Weyl group of $Spin(2n+1)$ 
and it acts on a typical element $(t_1,...,t_n)\in T,$ by permuting the entries
and changing the signs of even number of them.
We discuss two separate cases:\\
\textbf{Case 1:} $n$ is odd.

Let $t=(t_1,...,t_n)\in T$ be an arbitrary element of the torus. As in the case of $B_n,$ we consider a partition of $n$ as
 $n=n_1+...+n_k$, where the $n_i'$s are as in $\S 4.$ Thus looking at the torus element  $t,$ we can read off the isotropy subgroup, which 
is

$$
((\mathbb{Z}/2)^{n_1^\prime-1}\rtimes S_{n_1^\prime}))\times
 ((\mathbb{Z}/2)^{n_1-n_1^\prime-1}\rtimes S_{n_1-n_1^\prime})\times S_{n_2}\times...\times S_{n_k},
$$

Thus for each $n_1$ the number of non-conjugate isotropy subgroups is $([n_1/2]+1)p(n-n_1)$. This is because the 
number of partitions of $n_1$ which give non -conjugate isotropy subgroups for a fixed choice of $n_2,..,n_3$ is $[n_1/2].$ 
Hence the total number is
$$\sum_{i=0}^n ([i/2]+1)p(n-i).$$\\

\noindent\textbf{Case 2:} $n$ is even.

First let us investigate the following situation:
$t=(t_1,...,t_n)\in T$, where $t_1=...=t_{n-1}=-t_n$ and $t_i\neq 0,\pi,\pi/2,3\pi/2,$ for $1\leq i\leq n.$
We have the Weyl group $W=(\mathbb{Z}/2)^{n-1}\rtimes S_n$.
 The action of an element $(\tau,\rho)\in W$ on any $t\in T$ is given by,
$$
(\tau,\rho) (t_1,...,t_n) = (t_{(\rho)^{-1} (\tau)^{-1}(1)},...,t_{(\rho)^{-1} (\tau)^{-1}(n)}),
$$

If $(\tau,\rho)\in W_t,$ then $(\tau,\rho) (t_1,...,t_n) = (t_{(\rho)^{-1} (\tau)^{-1}(1)},...,t_{(\rho)^{-1} (\tau)^{-1}(n)})
=(t_1,...,t_n).$ Therefore, \\
\textbf{(a)} if $\rho (n)= n$ then $\tau=(0,...,0)\in (\mathbb{Z}/2)^{n-1}$\\
\textbf{(b)} if $\rho(n)=i \neq n$ then necessarily $\tau$ is an $n$-tuple with $1$ at the $n$-th and
$\rho(n)$-th positions and $0$ everywhere else.

The isotropy subgroup of $t$ therefore has exactly $n!$ many elements and as we will see, is not conjugate to
$S_n$ (since $S_n$ is the only other isotropy subgroup of order $n!$).

Let if possible $(\tau,\rho)\in W$ be such that $$(\tau,\rho)S_n(\tau,\rho)^{-1} = W_t.$$
Then, for an arbitrary $(1,\sigma)\in S_n\subset W$ we have,\\
\begin{align*}
&(\tau,\rho)(1,\sigma)(\rho^{-1}(\tau),\rho^{-1})\\
=&(\tau,\rho\sigma)(\rho^{-1}(\tau),\rho^{-1})\\
=&(\tau\rho\sigma\rho^{-1}(\tau),\rho\sigma\rho^{-1})\quad\in W_t.  
\end{align*} 
Note that $\tau$ cannot be $(0,...,0)$ or $(1,...,1)$ because in that case $\tau\rho\sigma\rho^{-1}(\tau)$ is necessarily equal
to $(0,...,0)$ for any chosen $\sigma$; and we can suitably choose a $\sigma\in S_n$ such that $\rho\sigma\rho^{-1}(n)\neq n$, in 
which case the above element cannot belong to $W_t$. Thus $\tau$ must contain both $0$ and $1$ as its parameters.
Moreover, since $(\tau,\rho)\in W$, $\tau$ must be a permutation changing an even number of signs. Since there is 
at least one $1$ in the $n$-tuple representing $\tau$, there must be at least two of them. Similar argument holds 
for the number of $0^,$s occurring in $\tau.$ Now let the $n$-th and the $i$-th positions in $\tau$ be $1.$ 
Then we simply choose a suitable $\sigma$
such that $\rho\sigma\rho^{-1}= (1~ n)$ (the transposition flipping $1$ and $n$).
 This shows that the element $(\tau\rho\sigma\rho^{-1}(\tau),\rho\sigma\rho^{-1})\notin
W_t$ because $\tau\rho\sigma\rho^{-1}(\tau)= (1,...1)$ in this case again.

With this in hand, we carry out the computation for the number of conjugacy classes in a way similar to that of $Spin(2n+1).$ 
If $n=n_1+...+n_k$ is a partition consisting of at least one odd integer, then
by the action of a suitable Weyl group element the computation can be carried out as in Case 1.

If the partition $n=n_1+...+n_k$ consists of only even integers, and also let us assume that none of the parameters are $0$ or $\pi,$ then
we can have the following possibility:\\
$t_1=...=t_{n_1-1}=-t_{n_1}$ and the remaining blocks containing equal parameters with $t_i\neq -t_j$ for $n_1< i,j\leq n.$ By the 
argument at the beginning of Case 2, the isotropy subgroup for such an element is obtained as: Let $n=2l.$
If $l=l_1+...+l_k$ , then
$W_t=H_{2k_1}.S_{2k_2}...S_{2k_l},$ where $H_{2k_1}$ is a subgroup of order $(2k_1)!$ as described in the beginning of Case 2.

So if $n=2l$ then the total number of conjugacy classes of isotropy subgroups is :
\begin{align*}
 &(\sum_{i=1}^n ([i/2]+1)p(n-i)) +p(n)-p(l)+2p(l)\\
=&(\sum_{i=0}^n ([i/2]+1)p(n-i)) +p(l).
\end{align*}

As noted in the previous section, over an algebraically closed field, the number
of conjugacy classes of isotropy subgroups of the Weyl group can be obtained exactly as above.
Thus we
have the following theorem:
\begin{theorem}
 The genus number of a compact simply connected Lie group
or a simply connected algebraic group over an algebraically closed field, of type $D_n$ 
 is\\
$\sum_{i=0}^n ([i/2]+1)p(n-i)$ for $n$ odd and\\
$\sum_{i=0}^n ([i/2]+1)p(n-i)) +p(l)$ for $n=2l$.
\end{theorem}
\begin{corollary}
 The connected genus number of $SO(2n)$ is equal to the genus number of $Spin(2n).$

\end{corollary}
\begin{proof}
 Follows from Theorem 2.3.
\end{proof}

.

\section{$F_4$}
Let $\mathfrak{C}$ be the octonion division algebra over $\mathbb{R}$ with norm $N.$ We fix an orthogonal basis $\mathfrak{B}
=\{v_1,v_2,...,v_8\} $, where $v_1=1$, $v_6=v_2v_5$, $v_7=v_3v_5$ and
$v_8=v_4v_5$ (\textbf{[P]}, Lecture 14). Let $Spin(N)$ and $SO(N)$ respectively denote the spin group and the special orthogonal group 
of $(\mathfrak{C},N).$ With respect to the basis $\mathfrak{B},$ the matrix of the bilinear form associated with $N$ is diagonal.

Consider the $\mathbb{R}$-algebra 
$A:= H_3(\mathfrak{C}),$ consisting of all $3\times 3$ matrices of the form
$\begin{bmatrix}
  \alpha_1&c_3&\bar{c_2}\\
\bar{c_3}&\alpha_2&c_1\\
c_2&\bar{c_1}&\alpha_3
 \end{bmatrix},$ where $\alpha_i\in \mathbb{R}$, $c_i\in \mathfrak{C}$ and $x\mapsto \bar{x}$ is the canonical involution on 
$\mathfrak{C}.$
The multiplication in $A$ is given by
$$xy=(x\cdot y+y\cdot x)/2,$$
where dot denotes the standard matrix multiplication and square is the usual one with respect to the matrix product.
 
Then $Aut(A)$ is the compact connected 
Lie group of type $F_4.$ For this discussion we need an explicit embedding of $Spin(N)$ in $F_4.$ Consider the subalgebra 
$S=\mathbb{R}\times\mathbb{R}\times \mathbb{R}\subset A$. Then $Spin(N)$ sits inside $Aut(A)$ as the subgroup of all automorphisms
$\phi,$
such that $\phi(s)=s$ for all $s\in S$ (\textbf{[J]}, Theorem 6).


We first discuss an explicit description of $Spin(N)$. Let as before $\mathfrak{C}$ denote an octonion algebra over 
$\mathbb{R}$ and consider a subgroup $RT(\mathfrak{C})\subset SO(N)^3,$
defined as,
\begin{align*}
 RT(\mathfrak{C}):=\{(t_1,t_2,t_3)\in SO(N)^3|~t_1(xy)=t_2(x)t_3(y)\quad \forall x,y\in \mathfrak{C}\}
\end{align*}
Any element of $ RT(\mathfrak{C})$ is called a related triple.
For related triples and Principle of triality one may refer to \textbf{[SV]},Chapter 3.
We need the following result from \textbf{[SV]} (Proposition 3.6.3).
\begin{proposition}
 There is an isomorphism,
$$\Phi:Spin(N)\longrightarrow RT(\mathfrak{C})$$
defined by ,
$$\Phi(a_1\circ b_1\circ...\circ a_r\circ b_r)=(s_{a_1}s_{b_1}...s_{a_r}s_{b_r},l_{a_1}l_{\overline{b_1}}...l_{a_r}l_{\overline{b_r}},
r_{a_1}r_{\overline{b_1}}...r_{a_r}r_{\overline{b_r}}),$$
where $a_i,b_i\in \mathfrak{C},\prod_iN(a_i)N(b_i)=1$,$($ $N$ being the norm on the octonion algebra$)$,
$s_v$ is the reflection in the hyperplane orthogonal to $v\in \mathfrak{C}$, $l_v$ and $r_v$ are the left and right homotheties
on $\mathfrak{C}$ respectively.
\end{proposition}
\textbf{Remark:}
Henceforth in the subsequent discussion we shall identify the groups $Spin(N)$ and $RT(\mathfrak{C})$ via the above isomorphism.
We note that a related triple $t=(t_1,t_2,t_3)\in RT(\mathfrak{C})$ acts on an element of $A$ as;
 $t\begin{bmatrix}
  \alpha_1&c_3&\bar{c_2}\\
\bar{c_3}&\alpha_2&c_1\\
c_2&\bar{c_1}&\alpha_3
 \end{bmatrix}=\begin{bmatrix}
  \alpha_1&t_1(c_3)&t_2(\bar{c_2})\\
\overline{t_1(c_3)}&\alpha_2&t_3(c_1)\\
\overline{t_2(\bar{c_2})}&\overline{t_3(c_1)}&\alpha_3
 \end{bmatrix}$(refer to \textbf{[J]}, $\S6$).

Consider the following automorphisms of $RT(\mathfrak{C})$:
\begin{align}{\label{a}}\tau_1:(t_1,t_2,t_3)\mapsto(\hat{t_1},\hat{t_3},\hat{t_2}),\nonumber \\
\tau_2:(t_1,t_2,t_3)\mapsto(t_3,\hat{t_2},t_1),\\
\tau_3:(t_1,t_2,t_3)\mapsto(t_2,t_1,\hat{t_3}),\nonumber
\end{align}
where $\hat{t}(x)=\overline{t(\overline{x})}$, for $t\in SO(N)$ and $x\in \mathfrak{C}$.
We note the following result from \textbf{[SV]} (Proposition 3.6.4),
\begin{proposition}{\label{outspin}}
 $\tau_2$ and $\tau_3$ generate a group of automorphisms of $RT(\mathfrak{C})$ isomorphic to $S_3$
 and the non trivial elements of this group are outer
automorphisms.
\end{proposition}.

\begin{lemma}{\label{mtspin}}
 Let $T$ be a maximal torus in $SO(N).$ Then $$\widetilde{T}:=\{(t_1,t_2,t_3)\in T^3|~(t_1,t_2,t_3)~ is~  a 
~ related ~ triple\}$$ is a maximal torus in $Spin(N).$
\end{lemma}
\begin{proof}
 If we take $t_1\in T,$ then the fiber of $t_1$ in a maximal torus
$\widetilde{T}$ of $Spin(N)$ consists of $(t_1,t_2,t_3)$ and $(t_1,-t_2,-t_3),$ such that 
$(t_1,t_2,t_3)$ is a related triple. Since the Weyl group acts on the maximal torus,
$\tau_3(t_1,t_2,t_3)=(t_2,t_1,\hat{t_3})\in \widetilde{T},$ which when projected onto $SO(N)$ via the two sheeted covering map,
we gives $t_2\in T$. Similarly by considering the automorphism $\tau_2$ we can conclude $t_3\in T.$ Hence the proof. 
\end{proof}

\begin{lemma}
 For a maximal torus $\widetilde{T}\subset F_4,$ $A^{\widetilde{T}}\cong \mathbb{R}\times\mathbb{R}\times \mathbb{R}.$ Here
 $A^{\widetilde{T}}$ denotes the subalgebra of $A,$ fixed point wise by $\widetilde{T}.$
 
\end{lemma}
\begin{proof}Let $T$ be the diagonal maximal torus of $SO(N).$
If $\widetilde{T_1}$ and $\widetilde{T_2}$ be two maximal tori in $F_4$,
then $A^{\widetilde{T}_1}\cong A^{\widetilde{T}_2}$ since $\widetilde{T_1}$ and $\widetilde{T_2}$ are conjugate. 
So we can assume without loss of
generality that, $\widetilde{T}\subset Spin(N)$ and hence by Lemma \ref{mtspin},
$\widetilde{T}=\{(t_1,t_2,t_3)\in Spin(N)|~t_i\in T\subset SO(N)\}.$
Now suppose $t\begin{bmatrix}
  \alpha_1&c_3&\bar{c_2}\\
\bar{c_3}&\alpha_2&c_1\\
c_2&\bar{c_1}&\alpha_3
 \end{bmatrix}=\begin{bmatrix}
  \alpha_1&t_1(c_3)&t_2(\bar{c_2})\\
\overline{t_1(c_3)}&\alpha_2&t_3(c_1)\\
\overline{t_2(\bar{c_2})}&\overline{t_3(c_1)}&\alpha_3
 \end{bmatrix}=\begin{bmatrix}
  \alpha_1&c_3&\bar{c_2}\\
\bar{c_3}&\alpha_2&c_1\\
c_2&\bar{c_1}&\alpha_3
 \end{bmatrix},$
holds for all $t\in \widetilde{T}.$ This means that $t_1(c_3)=c_3$ for all $t_1\in T.$ Note that
$t_1$ is a block diagonal matrix consisting of $2\times 2$ rotation matrices along the diagonal. Let if possible $c_3\neq 0.$
We can assume without loss of generality that at least one of the first two coordinates of $c_3$ (say $x_1,x_2$) with respect to the basis
 $\mathfrak{B}$ of $\mathfrak{C},$ is non zero.

Now if we take the first 
$2\times 2$ diagonal block of $t_1$ as $\begin{bmatrix}
                                        cos2\theta_1&-sin2\theta_1\\
                                         sin2\theta_1&cos2\theta_1
                                       \end{bmatrix},$ then $t_1(c_3)=c_3$ implies that
$\begin{bmatrix}
                                        cos2\theta_1&-sin2\theta_1\\
                                         sin2\theta_1&cos2\theta_1
                                       \end{bmatrix}\begin{bmatrix}
                                                     x_1\\
                                                      x_2
                                                    \end{bmatrix}=\begin{bmatrix}
                                                                  x_1\\
                                                                  x_2
                                                                 \end{bmatrix},$ which forces $cos2\theta_1=1.$ But we can choose a $t_1$
 with $\theta_1\neq 0,$ for which $cos2\theta_1\neq 1.$
Hence $c_3=0$. By similar arguments we can say the same 
for $c_1$ and $c_2.$ Hence the proof.
\end{proof}

\begin{lemma}
The Weyl group of $F_4$ is $WSpin(N)\rtimes S_3,$ $WSpin(N)$ being the Weyl group of $Spin(N)$.
\end{lemma}

\begin{proof}
Let us denote the group $F_4$ by $G$. Consider the $\mathbb{R}-$subalgebra $S=\mathbb{R}\times\mathbb{R}\times \mathbb{R}\subset A$ 
and define,
$$Aut(A/S):= \{\phi\in Aut(A):\phi(s)=s,~\forall s\in S\},$$
$$Aut(A,S):=\{\phi\in Aut(A):\phi(S)=S\}.$$
Then $Aut(A,S)\cong Aut(A/S)\rtimes Aut(S)$ (\textbf{[J]}, Theorem 8).
We have $Aut(A/S)= Spin(N)$ and $Aut(S)= S_3$ and therefore, $Aut(A,S)=Spin(N)\rtimes S_3$.

First let us fix a maximal torus $T\subset G$. Then $A^T\cong \mathbb{R}\times \mathbb{R}\times
\mathbb{R}$(by Lemma 7.2). Let $\phi\in N_G(T).$ Then $\phi\in Aut(A,A^T),$ since,
for $s\in A^T$  and for any $t\in T$ we have $
                                             t(\phi(s))=(t\phi)(s)
                                                       =\phi(\phi^{-1}t\phi)(s)
                                                       =\phi(s)$ (as $\phi^{-1}t\phi\in T$ and $s\in A^T$). Hence $\phi(s)\in A^T.$
Therefore we have shown that $N_G(T)\subset Aut(A,A^T)=Spin(N)\rtimes S_3.$ Thus $N_G(T)\subset N_{Spin(N)}(T)\rtimes S_3,$ which implies
that $WG=N_G(T)/T\subset WSpin(N)\rtimes S_3$. Both the groups being finite and of the same order, are therefore equal.

\end{proof} 
\noindent\textbf{Remark:}
 Note that, the $S_3$ factor 
arising in the Weyl group of $F_4$ is the group of outer automorphisms of $Spin(N)$ and its 
action on the maximal torus is given by $\tau_1,\tau_2,\tau_3\in Aut(RT(\mathfrak{C}))$ (refer to the remark preceding Proposition 
\ref{outspin}).

\noindent \textbf{Computation of the genus number for $F_4$:}

Let us denote the maximal torus in $F_4$ by $\widetilde{T}$ and the Weyl group by $W.$  
We work with the chosen orthogonal basis $\mathfrak{B}=\{v_1,...,v_8\}$ of $\mathfrak{C},$ 
such that, $v_1=1$, $v_6=v_2v_5$, $v_7=v_3v_5$
and $v_i^2=-1,$ $1\leq i\leq 8.$ Let $T\subset SO(N)$ be the diagonal maximal torus and without loss of generality we can assume
$\widetilde{T}\subset Spin(N).$
 If $t=(t_1,t_2,t_3)\in \widetilde{T}$, with 
$t_1=(\theta_1/\pi,\theta_2/\pi,\theta_3/\pi,\theta_4/\pi)$, $\theta_i/2\pi\in \mathbb{R/Z},$ 
we wish to compute $t_2$ and $t_3$ in terms of the 
$\theta_i^,$s.

First note that for $t=(\gamma_1/\pi,\gamma_2/\pi,\gamma_3/\pi,\gamma_4/\pi)\in T\subset SO(N)$,
 $\hat{t}=(-\gamma_1/\pi,\gamma_2/\pi,\gamma_3/\pi,\gamma_4/\pi)$. This
is evident from the following calculation:
Let $x=(x_1,...,x_8)\in \mathfrak{C},$ $x_i\in \mathbb{R}.$ Then $\bar{x}=(x_1,-x_2,...,-x_8)$ (considered as a column vector). 
By definition, $\hat{t}(x)=\overline{t(\bar{x})}.$
Now, $t=(\gamma_1/\pi,\gamma_2/\pi,\gamma_3/\pi,\gamma_4/\pi)$ is 
an $8\times 8$ block diagonal matrix with the $i$-th diagonal block being:
$\begin{bmatrix}
cos2\gamma_i&-sin2\gamma_i\\
sin2\gamma_i&cos2\gamma_i
 \end{bmatrix}$ (by the notation used in Section 4). 
Let $s=(-\gamma_1/\pi,\gamma_2/\pi,\gamma_3/\pi,\gamma_4/\pi)\in \widetilde{T}.$ Then, a direct computation shows that,
$$\hat{t}(x)=\overline{t(\bar{x})}=\begin{bmatrix}
cos2\gamma_1x_1+sin2\gamma_1x_2\\
-sin2\gamma_1x_1+cos2\gamma_1x_2\\
cos2\gamma_2x_3-sin2\gamma_2x_4\\
sin2\gamma_2x_3+cos2\gamma_2x_4\\
cos2\gamma_3x_5-sin2\gamma_3x_6\\
sin2\gamma_3x_5+cos2\gamma_3x_6\\
cos2\gamma_4x_7-sin2\gamma_4x_8\\
sin2\gamma_4x_7+cos2\gamma_4x_8
                        \end{bmatrix}=s(x).
$$ 
Therefore, \begin{align}{\label{b}}\hat{t}=(-\gamma_1/\pi,\gamma_2/\pi,\gamma_3/\pi,\gamma_4/\pi).\end{align}

If $t_1=(\theta_1/\pi,0,0,0)$ then a direct computation gives $t_1=s_as_b$, with $a=sin\theta_1v_1-cos\theta_1v_2$ and $b=v_2$.
We now calculate
$t_2$ and $t_3$. Recall that $t_1$ in matrix notation is an $8\times8$ matrix consisting of four $2\times2$ identity
diagonal blocks, the
first block being$$\begin{bmatrix}
                    cos2\theta_1&-sin2\theta_1\\
                    sin2\theta_1&cos2\theta_1
                   \end{bmatrix}$$
and $2\times2$ identity blocks in the next three diagonal positions.
So in order to calculate $t_2$ and $t_3$ we just evaluate these on the basis vectors, look at the matrices and get the parameters. We have,
\begin{align*}
 &l_al_{\overline{b}}(v_1)=a\overline{b}=(sin\theta_1v_1-cos\theta_1v_2)(-v_2)=-cos\theta_1v_1-sin\theta_1v_2\\
& l_al_{\overline{b}}(v_2)=a(\overline{b}v_2)=-a(v_2^2)=sin\theta_1v_1-cos\theta_1v_2\\
&l_al_{\overline{b}}(v_3)=a(\overline{v_2}v_3)=-av_4=-cos\theta_1v_3-sin\theta_1v_4\\
 &l_al_{\overline{b}}(v_4)=-a(v_2v_4)=av_3=sin\theta_1v_3-cos\theta_1v_4\\
 &l_al_{\overline{b}}(v_5)=-a(v_2v_5)=av_6=-cos\theta_1v_5-sin\theta_1v_6\\
&l_al_{\overline{b}}(v_6)=-a(v_2v_6)=av_5=sin\theta_1v_5-cos\theta_1v_6\\
& l_al_{\overline{b}}(v_7)=-a(v_2v_7)=av_8=-cos\theta_1v_7+sin\theta_1v_8\\
&l_al_{\overline{b}}(v_8)=-a(v_2v_8)=-a(v_7)=-sin\theta_1v_7-cos\theta_1v_8.
\end{align*}
This gives us $t_2.$ Next we compute $t_3$ as:
\begin{align*}
& r_ar_{\overline{b}}(v_1)=-v_2a=-cos\theta_1v_1-sin\theta_1v_2\\
&r_ar_{\overline{b}}(v_2)=-v_2^2a=sin\theta_1v_1-cos\theta_1v_2\\
& r_ar_{\overline{b}}(v_3)=-(v_3v_2)a=v_4a=-cos\theta_1v_3+sin\theta_1v_4\\
 &r_ar_{\overline{b}}(v_4)=-(v_4v_2)a =-v_3a=-sin\theta_1v_3-cos\theta_1v_4\\
& r_ar_{\overline{b}}(v_5)=-(v_5v_2)a=v_6a=-cos\theta_1v_5+sin\theta_1v_6\\
 &r_ar_{\overline{b}}(v_6)=-(v_6v_2)a=-v_5a=-sin\theta_1v_5-cos\theta_1v_6\\
&r_ar_{\overline{b}}(v_7)=-(v_7v_2)a-v_8a=-cos\theta_1v_7-sin\theta_1v_8\\
 &r_ar_{\overline{b}}(v_8)=-(v_8v_2)a=v_7a=sin\theta_1v_7-cos\theta_1v_8\\
\end{align*}
So $t_1,t_2,t_3$ in their possible parametric forms are given as follows:
\begin{align*}
 &t_1=(\theta_1/\pi,0,0,0)\\
&t_2=((\pi+\theta_1)/2\pi,(\pi+\theta_1)/2\pi,(\pi+\theta_1)/2\pi,-(\pi+\theta_1)/2\pi)\\
&t_3=((\pi+\theta_1)/2\pi,-(\pi+\theta_1)/2\pi,-(\pi+\theta_1)/2\pi,(\pi+\theta_1)/2\pi)
\end{align*}
\begin{align*}
 &t_1=(0,\theta_2/\pi,0,0)\\
&t_2=((\pi+\theta_2)/2\pi,(\pi+\theta_2)/2\pi,-(\pi+\theta_2)/2\pi,(\pi+\theta_2)/2\pi)\\
&t_3=(-(\pi+\theta_2)/2\pi,(\pi+\theta_2)/2\pi,-(\pi+\theta_2)/2\pi,(\pi+\theta_2)/2\pi)
\end{align*}
\begin{align*}
 &t_1=(0,0,\theta_3/\pi,0)\\
&t_2=((\pi+\theta_3)/2\pi,-(\pi+\theta_3)/2\pi,(\pi+\theta_3)/2\pi,(\pi+\theta_3)/2\pi)\\
&t_3=(-(\pi+\theta_3)/2\pi,-(\pi+\theta_3)/2\pi,(\pi+\theta_3)/2\pi,(\pi+\theta_3)/2\pi)
\end{align*}
\begin{align*}
 &t_1=(0,0,0,\theta_4/\pi)\\
&t_2=(-(\pi+\theta_4)/2\pi,(\pi+\theta_4)/2\pi,(\pi+\theta_4)/2\pi,(\pi+\theta_4)/2\pi)\\
&t_3=((\pi+\theta_4)/2\pi,(\pi+\theta_4)/2\pi,(\pi+\theta_4)/2\pi,(\pi+\theta_4)/2\pi)
\end{align*}
Therefore in general we have,
\begin{align*}
 &t_1=(\theta_1/\pi,\theta_2/\pi,\theta_3/\pi,\theta_4/\pi)\\
&t_2=((\theta_1+\theta_2+\theta_3-\theta_4)/2\pi,(\theta_1+\theta_2-\theta_3+\theta_4)/2\pi,
(\theta_1-\theta_2+\theta_3+\theta_4)/2\pi,(-\theta_1+\theta_2+\theta_3+\theta_4)/2\pi)\\
&t_3=((\theta_1-\theta_2-\theta_3+\theta_4)/2\pi,(-\theta_1+\theta_2-\theta_3+\theta_4)/2\pi,
(-\theta_1-\theta_2+\theta_3+\theta_4)/2\pi,(\theta_1+\theta_2+\theta_3+\theta_4)/2\pi)
\end{align*}
We record the above set of equations as $(\ast).$ These parameters are written modulo $\mathbb{Z}.$
Now we analyse all the possibilities for $\theta_i^,$s to compute the non conjugate isotropy classes.

\noindent\textbf{Case1:(At least one $\theta_i$ is $0$ or $1/2$)}

\noindent\textbf{(a)}  If $\theta_i=0\quad \forall i$, then by $(\ast),$ $t_1=t_2=t_3=(0,0,0,0)$ and hence $W_t=W.$

\noindent\textbf{(b)}  If $\theta_i/\pi=1/2\quad \forall i$, then by $(\ast),$ we have,\\
$t_1=t_2=(1/2,1/2,1/2,1/2)$ and $t_3=(0,0,0,0)$. Note that only $\tau_3$ from $S_3=Out(Spin(N))$ occurs in 
the stabilizer  since it leaves
$t$ stable and any other element from $S_3$ brings $t_3$ in the first place from which we cannot get back $t_1$ by the 
action of any element from $WSpin(N)$ (see \ref{a}, \ref{b}). Thus $W_t=((\mathbb{Z}/2)^3\rtimes S_4)\rtimes\{1,\tau_3\}$.

\noindent\textbf{(c)}  If $t_1=(0,0,0,1/2)$, then by $(\ast),$
\begin{align*}
 &t_2=(-1/4,1/4,1/4,1/4)\\
&t_3=(1/4,1/4,1/4,1/4)
\end{align*}
Note here that $\tau_1(t)=t$ and hence $\tau_1\in W_t$ and no other element from $S_3$ can occur because $t_1$ has $0^,$s
as parameters but $t_2,t_3$ do 
not (see \ref{a}, \ref{b}). Hence $W_t=((\mathbb{Z}/2)^2\rtimes S_3)\rtimes\{1,\tau_1\}$.
 
\noindent\textbf{(d)}  If 
                $t_1=(1/2,1/2,1/2,0),$ then by $(\ast),$
\begin{align*}
&t_2=(3/4,1/4,1/4,1/4)\\
&t_3=(3/4,3/4,3/4,1/4)
               \end{align*}
Here $W_t=((\mathbb{Z}/2)^2\rtimes S_3),$ because any element from $Out(Spin(N))$ will  alter $t_2,t_3$ and as a result we cannot get
back $t$ by a subsequent action of $WSpin(N)$ (see \ref{a}, \ref{b}). 

\noindent\textbf{(e)}  If $t_1=(0,0,1/2,1/2)$ then
by $(\ast),$ $t_1=t_2=t_3$ and the isotropy is $((\mathbb{Z}/2\rtimes S_2)\times(\mathbb{Z}/2))\rtimes S_2)
\rtimes S_3$.

\noindent\textbf{(f)}  If $t_1=(0,0,0,\theta_4/\pi)$ with $\theta_4/\pi\neq 0,1/2$ then by $(\ast),$
\begin{align*}
t_2=(-\theta_4/2\pi,\theta_4/2\pi,\theta_4/2\pi,\theta_4/2\pi)\\
t_3=(\theta_4/2\pi,\theta_4/2\pi,\theta_4/2\pi,\theta_4/2\pi).
\end{align*}
In this case apart from $\tau_1$ no other element from $S_3$ can contribute to the isotropy since $t_1$ contains $0$ and $t_2,t_3$
do not (see \ref{a}. \ref{b}). So $W_t=((\mathbb{Z}/2)^2\rtimes S_3)\rtimes\{1,\tau_1\}$, being same as case (\textbf{c}).

\noindent\textbf{(g)} If 
              $t_1=(1/2,1/2,1/2,\theta_4/\pi),$ then by $(\ast),$
\begin{align*}
&t_2=(3/4-\theta_4/2\pi,1/4+\theta_4/2\pi,1/4+\theta_4/2\pi,1/4+\theta_4/2\pi)\\
&t_3=(-1/4+\theta_4/2\pi,-1/4+\theta_4/2\pi,-1/4+\theta_4/2\pi,3/4+\theta_4/2\pi)
             \end{align*}
Here, just as in \textbf{(d)}, we have  $W_t=((\mathbb{Z}/2)^2\rtimes S_3)\subset Spin(N).$ 

\noindent\textbf{(h)} If
              $t_1=(0,0,\theta/\pi,\theta/\pi),$ then by $(\ast),$ $t_1=t_2=t_3.$
             
Clearly here, the whole of $S_3$ leaves $t$ stable (by \ref{a}, \ref{b}) and hence $W_t=((\mathbb{Z}/2\rtimes S_2)\times S_2)\rtimes S_3.$

\noindent\textbf{(i)}  If
               $t_1=(1/2,1/2,\theta/\pi,\theta/\pi),$ then by $(\ast),$
\begin{align*}
&t_2=(1/2,1/2,\theta/\pi,\theta/\pi)\\
&t_3=(0,0,1/2+\theta/\pi,1/2+\theta/\pi)
              \end{align*}
Now $(t_1,t_2,t_3)=\tau_2(s_1,s_2,s_3)=(s_3,\hat{s_2},s_1),$ (by \ref{b}) where,
\begin{align*}
 &s_1=(0,0,1/2+\theta/\pi,1/2+\theta/\pi)\\
&s_2=(1/2,1/2,\theta/\pi,\theta/\pi)\\
&s_3=(1/2,1/2,\theta/\pi,\theta/\pi).
\end{align*} 
If $s=(s_1,s_2,s_3),$
 $W_t$ is conjugate to $W_s$ in $W.$
Since any element of $S_3$ other than $\tau_1$ removes $s_1$ from the first position, $\tau_1$ is the only element from $S_3$ which contributes
to the isotropy of $s$ (see \ref{a}) Hence $W_s=((\mathbb{Z}/2\rtimes S_2)\times S_2)\rtimes \{1,\tau_1\}.$

\noindent\textbf{(j)}  If
               $t_1=(0,\theta/\pi,\theta/\pi,\theta/\pi),$ then by $(\ast),$
\begin{align*}
&t_2=(\theta/2\pi,\theta/2\pi,\theta/2\pi,3\theta/2\pi)\\
&t_3=(-\theta/2\pi,\theta/2\pi,\theta/2\pi,3\theta/2\pi)
\end{align*}
Here $\tau_1(t)=t$ and no other element from $S_3= Out(Spin(N))$ can contribute to the isotropy,
since $t_1$ has a $0$ and $\hat{t_2}=t_3$ (\ref{a}. \ref{b}). Thus $W_t=S_3\rtimes\{1,\tau_1\}.$

\noindent\textbf{(k)} If 
               $t_1=(1/2,\theta/\pi,\theta/\pi,\theta/\pi),$ then by $(\ast),$
\begin{align*}
&t_2=(1/4+\theta/2\pi,1/4+\theta/2\pi,1/4+\theta/2\pi,-1/4+3\theta/2\pi)\\
&t_3=(1/4-\theta/2\pi,-1/4+\theta/2\pi,-1/4+\theta/2\pi,1/4+3\theta/2\pi).
              \end{align*}
Here,  $\theta/\pi\neq0,1/2.$ Therefore $t_2,t_3$ does not contain $0$ or $1/2$ as parameters. Hence, $\tau_2,\tau_3\in S_3$ 
does not contribute to the isotropy.
 As $t_2\neq \hat{t_3},$ $\tau_1\in S_3$ 
cannot belong to the isotropy (see \ref{a}, \ref{b}).
 Therefore, $W_t=S_3\subset WSpin(N).$

\noindent\textbf{(l)}  If
               $t_1=(0,0,\theta_3/\pi,\theta_4/\pi),$ then by $(\ast),$
\begin{align*}
&t_2=((\theta_3-\theta_4)/2\pi,(-\theta_3+\theta_4)/2\pi,(\theta_3+\theta_4)/2\pi,(\theta_3+\theta_4)/2\pi)\\
&t_3=((-\theta_3+\theta_4)/2\pi,(-\theta_3+\theta_4)/2\pi,(\theta_3+\theta_4)/2\pi,(\theta_3+\theta_4)/2\pi)
              \end{align*}
We assume here $\theta_3/\pi\neq \theta_4/\pi$ modulo $\mathbb{Z}.$
Therefore $0$ does not occur in $t_2$ and $t_3$, so the only non trivial element from
$S_3$ which lies in the isotropy is $\tau_1$ (see \ref{a}, \ref{b}). Thus, $W_t=(\mathbb{Z}/2\rtimes S_2)\rtimes \{1,\tau_1\}$

\noindent\textbf{(m)}  If $
               t_1=(1/2,1/2,\theta_3/\pi,\theta_4/\pi),$ then by $(\ast),$
\begin{align*}
&t_2=(1/2+(\theta_3-\theta_4)/2\pi,1/2+(\theta_4-\theta_3)/2\pi,(\theta_3+\theta_4)/2\pi,(\theta_3+\theta_4)/2\pi)\\
&t_3=((\theta_4-\theta_3)/2\pi,(\theta_4-\theta_3)/2\pi,1/2+(\theta_3+\theta_4)/2\pi,1/2+(\theta_3+\theta_4)/2\pi)
              \end{align*}
Here $\hat{t_3}\neq t_2$ and $\hat{t_2}\neq t_3$ and $t_1$, contains $1/2$ as a parameter. So $S_3=Out(Spin(N))$
does not contribute to the isotropy (see \ref{a}, \ref{b}). Hence $W_t=\mathbb{Z}/2\rtimes S(2).$

\noindent\textbf{(n)} If $
               t_1=(0,\theta/\pi,\theta/\pi,\theta_4/\pi),$ then by $(\ast),$
\begin{align*}
&t_2=((2\theta-\theta_4)/2\pi,\theta_4/2\pi,\theta_4/2\pi,(2\theta+\theta_4)/2\pi)\\
&t_3=((-2\theta+\theta_4)/2\pi,\theta_4/2\pi,\theta_4/2\pi,(2\theta+\theta_4)/2\pi).
              \end{align*}
We have $W_t=S_2\rtimes\{1,\tau_1\}$
in this case, because again $\hat{t_2}=t_3$ and $\hat{t_3}=t_2$. And if $\theta/\pi=\theta_4/2\pi,$ we have by $(\ast) ,$
$t_1=t_2=t_3$ and $W_t=S_2\rtimes S_3$ (see \ref{a}, \ref{b}).

\noindent\textbf{(o)}  If $
               t_1=(1/2,\theta/\pi,\theta/\pi,\theta_4/\pi),$ then by $(\ast),$
\begin{align*}
&t_2=(1/4+(2\theta-\theta_4)/2\pi,1/4+\theta_4/2\pi,1/4+\theta_4/2\pi,-1/4+(2\theta+\theta_4)/2\pi)\\
&t_3=(1/4+(-2\theta+\theta_4)/2\pi,-1/4+\theta_4/2\pi,-1/4+\theta_4/2\pi,1/4+(2\theta+\theta_4)/2\pi).
              \end{align*}
Here $W_t=S_2\subset WSpin(N)$ because no element from $S_3$ can contribute to the isotropy of this element, as
we have taken $\theta/\pi\neq \theta_4/\pi$ and  hence $1/2$ does not occur in $t_2$ and $t_3$ (see \ref{a}, \ref{b}).

\noindent\textbf{(p)} If $
               t_1=(0,\theta_2/\pi,\theta_3/\pi,\theta_4/\pi),$ then by $(\ast),$
\begin{align*}
&t_2=((\theta_2+\theta_3-\theta_4)/2\pi,(\theta_2-\theta_3+\theta_4)/2\pi,(-\theta_2+\theta_3+\theta_4)/2\pi,(\theta_2+\theta_3+\theta_4)/2\pi)\\
&t_3=((-\theta_2-\theta_3+\theta_4)/2\pi,(\theta_2-\theta_3+\theta_4)/2\pi,(-\theta_2+\theta_3+\theta_4)/2\pi,(\theta_2+\theta_3+\theta_4)/2\pi).
              \end{align*}
If none of the coordinates in $t_2,t_3$ are $0,1/2$ then $W_t=\{1,\tau_1\}$, otherwise the only non trivial possibility 
is $W_t=S_3\subset WSpin(N),$
which occurs if $(\theta_2+\theta_3)/\pi=\theta_4/\pi,$ in which case $t_1=t_2=t_3$ holds by $(\ast)$ (refer to \ref{a}, \ref{b}).

\noindent\textbf{Case 2:(no $\theta_i$  in $t_1$ are $0,1/2$)} Here, however the isotropy subgroups for various possibilities for
$\theta_i$ are conjugate to certain subgroups already occurring in Case 1, except the situation when all $\theta_i^,$s are distinct,
which yields the trivial isotropy subgroup.

\noindent\textbf{(a)} If $ 
             t_1=(\theta/\pi,\theta/\pi,\theta/\pi,\theta/\pi),$ then by $(\ast),$
\begin{align*}
&t_2=(\theta/\pi,\theta/\pi,\theta/\pi,\theta/\pi)\\
&t_3=(0,0,0,2\theta/\pi)
            \end{align*}
 Then clearly $W_t=S_4\rtimes \{1,\tau_3\}$ since $\tau_3$ contributes to the isotropy from $S_3$ (see \ref{a}, \ref{b}) 
 and this isotropy is conjugate to that in  case 1(c).

\noindent\textbf{(b)} If $ 
              t_1=(\theta_1/\pi.\theta_1/\pi,\theta_2/\pi,\theta_2/\pi),$ then by $(\ast),$
\begin{align*}
&t_2=(\theta_1/\pi.\theta_1/\pi,\theta_2/\pi,\theta_2/\pi)\\
&t_3=(0,0,(\theta_2-\theta_1)/\pi,(\theta_1+\theta_2)/\pi).
              \end{align*}
Note that, $(t_1,t_2,t_3)=\tau_2(s_1,s_2,s_3),$ where,
\begin{align*}
 &s_1=(0,0,(\theta_2-\theta_1)/\pi,(\theta_1+\theta_2)/\pi)\\
&s_2=(-\theta_1/\pi.\theta_1/\pi,\theta_2/\pi,\theta_2/\pi)\\
&s_3=(\theta_1/\pi,\theta_1/\pi,\theta_2/\pi,\theta_2/\pi)
\end{align*}
which case has already been considered before (case 1(l)).

\noindent\textbf{(c)}  If 

 $t_1=(\theta_1/\pi,\theta_1/\pi,\theta_3/\pi,\theta_4/\pi),$ then by $(\ast),$
\begin{align*}
&t_2=((2\theta_1+\theta_3-\theta_4)/2\pi,(2\theta_1-\theta_3+\theta_4)/2\pi,(\theta_3+\theta_4)/2\pi,(\theta_3+\theta_4)/2\pi)\\
&t_3=((\theta_4-\theta_3)/2\pi,(\theta_4-\theta_3)/2\pi,(-2\theta_1+\theta_3+\theta_4)/2\pi,(2\theta_1+\theta_3+\theta_4)/2\pi)
\end{align*}
If $\theta_1/\pi\neq (\theta_3+\theta_4)/2\pi$
or $\theta_1/\pi\neq (\theta_4-\theta_3)/2\pi$
modulo $\mathbb{Z},$
then $W_t=S_2$(which has already occurred in case (o) of case 1).
 If $\theta_1/\pi$ is equal to any one of the above two elements (modulo $\mathbb{Z}$) then  $t_2$ or $t_3$ has 
  $0$ as one of it's co-ordinates. Accordingly $t_2$ or $t_3$ can be brought to the first position of the related triple (see \ref{a}).
Note that for all related triples $(t_1,t_2,t_3)$ such that $t_1$ has at least one $0$ as a parameter, the isotropy subgroups
have been computed in Case 1. Hence, this does not give us any new isotropy subgroup.

Now we consider $(t_1,t_2,t_3)$ such that $t_i$ has all the parameters distinct and not equal to zero. For this situation we record the
following lemmas.

\begin{lemma}{\label{fp}}
 If $t_i\in SO(N)$ does not have any of the parameters equal to zero, then $\mathfrak{C}^{t_i}=\{0\}$.
\end{lemma}
\begin{proof}
 Let $x\in \mathfrak{C}^{t_i}$ with $x\neq 0$ for some $i$.
Without loss of generality we can assume that $x_1\neq 0$, where $x_1$ denotes the first coordinate 
of $x$ with respect to the chosen basis $\mathfrak{B}=\{v_1,...,v_8\}.$ Hence the first $2\times2$ block
$$\begin{bmatrix} 
   
                    cos2\theta_1&-sin2\theta_1\\
                    sin2\theta_1&cos2\theta_1
                   
  \end{bmatrix}$$
of $t_1$ has a non zero eigenvector $(x_1,x_2)$ which implies that $\theta_1/\pi=0$, which is a contradiction
to the assumption that no parameter of $t_i$ is $0.$
\end{proof}
An element $x$ in a connected group $G$ is called strongly regular if $Z_G(t)=T.$
\begin{lemma}{\label{strreg}}
 If $t_1\in SO(N)$ be strongly regular then $(t_1,t_2,t_3)$ is strongly regular in $Spin(N).$ 
\end{lemma}
\begin{proof}
 Let $t_1\in SO(N)$ be strongly regular and $T\subset SO(N) $ be the maximal torus containing $t_1.$ Then $Z_{SO(N)}(t_1)=T.$
Let $s=(s_1,s_2,s_3)\in Spin(N)$ and $st=ts.$
Therefore,\\ 
           $s_1t_1=t_1s_1
\Rightarrow s_1\in T
\Rightarrow s_2,s_3\in T
\Rightarrow (s_1,s_2,s_3)\in \widetilde{T}$ (by Lemma \ref{mtspin})
$\Rightarrow Z_{Spin(N)}(t)=\widetilde{T}.$
Hence $(t_1,t_2,t_3)$ is strongly regular in $Spin(N).$
          
\end{proof}

\begin{theorem}{\label{strreg2}}
 If $t_i$ does not have any parameter equal to $0,$ and all parameters in $t_i$ are distinct, $1\leq i\leq 3,$
 then $(t_1,t_2,t_3)$ is strongly regular in $F_4$ 
and hence $W_t=\{1\}.$
\end{theorem}
\begin{proof}
Since $t_i$ does not have $0$ for all $i$, by Lemma \ref{fp}, $\mathfrak{C}^{t_i}=\{0\}$ $\forall i.$ Hence by this and the 
remark preceding 
Proposition 8.2,
$A^t=\mathbb{R}\times\mathbb{R}\times\mathbb{R}.$
So if $\phi\in Z_{F_4}(t),$
then $\phi(\mathbb{R}\times\mathbb{R}\times\mathbb{R})=\mathbb{R}\times\mathbb{R}\times\mathbb{R}$\\
$\Rightarrow \phi\in Aut(A,\mathbb{R}\times\mathbb{R}\times\mathbb{R})\cong Spin(N)\rtimes S_3$(by \textbf{[J]}, Theorem 8.)\\
$\Rightarrow Z_{F_4}(t)\subset Spin(N)\rtimes S_3$\\
$\Rightarrow Z_{F_4}(t)\subset Spin(N)$ (since $F_4$ is simply connected, $Z_{F_4}(t)$ is connected by Proposition \ref{conncent1}).\\
$\Rightarrow Z_{F_4}(t)\subset Z_{Spin(N)}(t).$ 

Since all parameters of $t_1$ are distinct and none of them is $0,$ 
the isotropy subgroup of $t_1$ in $WSO(N)$ is trivial. Note that $WSO(N)_{t_1}=Z_{SO(N)}(t_1)/T,$ where $T$ is the diagonal maximal torus in
$SO(N).$ Therefore, $WSO(N)_{t_1}=\{1\}\Rightarrow Z_{SO(N)}(t_1)=T,$ which means $t_1$  is strongly regular in $SO(N).$ Hence by
Lemma \ref{strreg}, $t=(t_1,t_2,t_3)$ is strongly regular in $Spin(N).$ Therefore, $Z_{F_4}(t)\subset Z_{Spin(N)}(t)=\widetilde{T}.$ This
is in fact an equality since, $\widetilde{T}\subset Z_{F_4}(t)$ for all $t\in \widetilde{T}.$ Thus $t$ is strongly regular in $F_4.$
 \end{proof}

We now proceed to calculate the semisimple genus number of a connected
algebraic group of type $F_4$ over an algebraically closed field $k$ of characteristic different from $2.$
Let $\mathfrak{C}$ and $\mathbb{H}$ be respectively the (split) octonion and quaternion algebras over
$k$, i.e. $\mathfrak{C}:=\mathbb{H}\oplus \mathbb{H},$ where
$$
\mathbb{H}:=\{ \begin{bmatrix}
                a&b\\
                c&d
               \end{bmatrix} : a,b,c,d \in k \},$$
under the usual matrix addition and multiplication with the norm $N:H\rightarrow k,$ defined as $N(x)=det(x),$ for $x\in \mathbb{H}.$ 
The norm for $\mathfrak{C}$ is given by $N((x,y))=det(x)-det(y),$ for $x,y\in \mathbb{H}$.
The conjugation in $\mathbb{H}$ is given by $$\overline{\begin{bmatrix}
                a&b\\
                c&d
               \end{bmatrix} }=\begin{bmatrix}
                d&-b\\
                -c&a
               \end{bmatrix}.$$
The multiplication and conjugation in $\mathfrak{C}$ are 
as follows:
\begin{align*}
 &(x,y)(u,v):=(xu+\bar{v}y,vx+y\bar{u}),\\
&\overline{(x,y)}:=(\bar{x},-y),
\end{align*}
where $x,y,u,v\in \mathbb{H}$.\\
We consider the following basis $\{ v_1,...,v_8\}$ of $\mathfrak{C}$:-\\
$$v_1=( \begin{bmatrix}
                1&0\\
                0&0
               \end{bmatrix},0),~ v_2=(\begin{bmatrix}
                0&-1\\
                0&0
               \end{bmatrix},0),~v_3=(0,\begin{bmatrix}
                -1&0\\
                0&0
               \end{bmatrix}),~v_4=(0,\begin{bmatrix}
                0&1\\
                0&0
               \end{bmatrix}),$$
$$v_5=(0,\begin{bmatrix}
                0&0\\
                1&0
               \end{bmatrix}),~v_6=(0,\begin{bmatrix}
                0&0\\
                0&1
               \end{bmatrix}),~v_7=(\begin{bmatrix}
                0&0\\
                1&0
               \end{bmatrix},0),~v_8=(\begin{bmatrix}
                0&0\\
                0&1
               \end{bmatrix},0).$$
The multiplication table for $\mathfrak{C}$ with respect to this basis is:\\

\begin{center}
\begin{tabular}{|l|l|l|l|l|l|l|l|l|}
\hline
$\cdotp$&$v_1$&$v_2$&$v_3$&$v_4$&$v_5$&$v_6$&$v_7$&$v_8$\\ \hline
$v_1$&$v_1$&$v_2$&$v_3$&$0$&$v_5$&$0$&$0$&$0$\\ \hline
$v_2$&$0$&$0$&$v_4$&$0$&$-v_6$&$0$&$-v_1$&$v_2$\\ \hline
$v_3$&$0$&$-v_4$&$0$&$0$&$v_7$&$-v_1$&$0$&$v_3$\\ \hline
$v_4$&$v_4$&$0$&$0$&$0$&$-v_8$&$-v_2$&$v_3$&$0$\\ \hline
$v_5$&$0$&$v_6$&$-v_7$&$-v_1$&$0$&$0$&$0$&$v_5$\\ \hline
$v_6$&$v_6$&$0$&$-v_8$&$v_2$&$0$&$0$&$-v_5$&$0$\\ \hline
$v_7$&$v_7$&$-v_8$&$0$&$-v_3$&$0$&$v_5$&$0$&$0$\\ \hline
$v_8$&$0$&$0$&$0$&$v_4$&$0$&$v_6$&$v_7$&$v_8$\\ \hline
\end{tabular}\\
\end{center}

With respect to the above  basis of $\mathfrak{C}$ the matrix of the bilinear form for the norm $N$ is 

$$\begin{bmatrix}
 {}&{}&1\\
{}&\ddots&{}\\
1&{}&{}\\
\end{bmatrix}$$ and
$$T:=\{diag(a,b,c,d,1/d,1/c,1/b,1/a)\in SO(N)|a,b,c,d\in k^\ast \}\subset SO(N)$$ is a maximal torus.
With the notation used for compact $F_4,$ any element of $Spin(N)$ corresponds 
uniquely to $(t_1,t_2,t_3)\in SO(N)^3$ such that 
$t_1(xy)=t_2(x)t_3(y)$ for all $x,y\in \mathfrak{C}$.\\
Let $t_1=diag(a,b,c,d,1/d,1/c,1/b,1/a)\in T.$ We can write
$t_1=s_{x_1}s_{y_1}...s_{x_4}s_{y_4},$  where $s_{x_i}$ denotes the reflection in the hyperplane perpendicular to $x_i$ and
 \begin{align*}
 &x_1=\sqrt{a}v_1+{\sqrt{a}}^{-1}v_8 ,~ y_1=v_1+v_8,
~x_2=\sqrt{b}v_2+{\sqrt{b}}^{-1}v_7,~ y_2=v_2+v_7\\
&x_3=\sqrt{c}v_3+{\sqrt{c}}^{-1}v_6,~ y_3=v_3+v_6,
~x_4=\sqrt{d}v_4+{\sqrt{d}}^{-1}v_5,~ y_4=v_4+v_5.
\end{align*}
Therefore, by Proposition 7.1, the corresponding $t_2,t_3$ are given by $t_2=l_{x_1}l_{\bar{y_1}}...l_{x_4}l_{\bar{y_4}}$  and\\
$t_3=r_{x_1}r_{\bar{y_1}}...r_{x_4}r_{\bar{y_4}}.$ So if we calculate $t_2$ and $t_3$ using  these formulas and the above multiplication
table we get (henceforth we shall denote an $8\times 8$ diagonal matrix of the form $diag(a,b,c,d,1/d,1/c,1/b,1/a)$ by
$(a,b,c,d)$ ),
\begin{align*}
 &t_1=(a,b,c,d),\\
&t_2=(\sqrt{a}\sqrt{b}\sqrt{c}/\sqrt{d},\sqrt{a}\sqrt{b}\sqrt{d}/\sqrt{c},\sqrt{a}\sqrt{c}\sqrt{d}/
\sqrt{b},\sqrt{b}\sqrt{c}\sqrt{d}/\sqrt{a}),\\
&t_3=(\sqrt{a}\sqrt{d}/\sqrt{b}\sqrt{c},\sqrt{b}\sqrt{d}/\sqrt{a}\sqrt{c},\sqrt{c}\sqrt{d}/\sqrt{a}\sqrt{b},
\sqrt{a}\sqrt{b}\sqrt{c}\sqrt{d}).
\end{align*}
Let us denote the above equations by $(\ast \ast).$

Now we can compute the isotropy classes in the Weyl group with respect to a maximal torus in $F_4$.
Let $T$ denote the diagonal maximal torus in $SO(N).$ Since any a maximal torus of $F_4$ sits inside a
copy of $Spin(N)\subset F_4$, we may work with 
$\widetilde{T}:=\{(t_1,t_2,t_3)\in T^3|t_1(xy)=t_2(x)t_3(y),\forall x,y\in T\}\subset RT(\mathfrak{C})\cong Spin(N).$

With this we can compute the isotropy subgroups of the Weyl group (the action of the Weyl group on the torus had already
been discussed before and we shall follow the same notations here). Recall that $W=((\mathbb{Z}/2)^3\rtimes S_4)\rtimes S_3$ is the Weyl
group of $F_4.$
In all the following cases the arguments for $W_t$ are exactly similar to the ones we had in the case for compact $F_4,$ only the roles 
played by $0$ and $1/2$ are replaced by $1$ and $-1$ respectively. With each of the following possibilities we refer to the corresponding
calculation done in the discussion on compact $F_4.$ In what follows, we denote a fixed square root of $-1$ by $i.$

\noindent\textbf{1.} $t_1=(1,1,1,1)=t_2=t_3$. In this situation clearly $W_t=W$ (case 1(a)).\\
\noindent\textbf{2.}\begin{align*}
            &t_1=t_2=(-1,-1,-1,-1)\\
&t_3=(1,1,1,1)
           \end{align*}
$W_t=((\mathbb{Z}/2)^3\rtimes S_4)\rtimes \{1,\tau_3\}$ (case 1(b)).

\noindent\textbf{3.}\begin{align*} 
&t_1=(1,1,1,-1)\\
&t_2=(-i,i,i,i)\\
&t_3=(i,i,i,i)
\end{align*}
$W_t=((\mathbb{Z}/2)^2\rtimes S_3)\rtimes\{1,\tau_1\}$ (case 1(c)).

\noindent\textbf{4.}\begin{align*}
            t_1=t_2=t_3=(1,1,-1,-1)
           \end{align*}
Note that all elements of $S_3$ fix this element $t$ and hence we have 
$W_t=(((\mathbb{Z}/2)\rtimes S_2)\times((\mathbb{Z}/2)\rtimes S_2))\rtimes S_3$ (case 1(e)).

\noindent\textbf{5.}\begin{align*}
            &t_1=(-1,-1,-1,1)\\
&t_2=(-i,i,i,i)\\ 
&t_3=(-i,-i,-i,-i)
           \end{align*}
Clearly no element from $S_3$ can belong to the isotropy, therefore $W_t=(\mathbb{Z}/2)^2\rtimes S_3.$ (case 1(d)).

\noindent\textbf{6.}$t_1=t_2=t_3=(1,1,c,c),$ where $c\neq 1,-1.$ Since any $S_3$ element leaves this fixed, we have
$W_t=((\mathbb{Z}/2\rtimes S_2)\times S_2)\rtimes S_3$ (case 1(h)).

\noindent\textbf{7.}\begin{align*}
            &t_1=t_2=(-1,-1,c,c)\\
&t_3=(1,1,c,c)
           \end{align*}
Here we observe that only $\tau_3\in S_3$ can contribute to the isotropy.
Hence $W_t=((\mathbb{Z}/2\rtimes S_2)\times S_2)\rtimes \{1,\tau_3\}$ (case 1(i)).

\noindent\textbf{8.}\begin{align*}
            &t_1=(1,b,b,b)\\
&t_2=(\sqrt{b},\sqrt{b},\sqrt{b},b\sqrt{b})\\
&t_3=(1/\sqrt{b},\sqrt{b},\sqrt{b},b\sqrt{b})
           \end{align*}
For this $W_t=S_3\rtimes\{1,\tau_1\}$ ( case 1(j)).

\noindent\textbf{9.}\begin{align*}
            &t_1=(-1,b,b,b)\\
&t_2=(i\sqrt{b},i\sqrt{b},i\sqrt{b},-ib\sqrt{b})\\
&t_3=(i/\sqrt{b},-i\sqrt{b},-i\sqrt{b},ib\sqrt{b})
           \end{align*}
where $b\neq1,-1$. $W_t=S_3$ (case 1(k)).

\noindent\textbf{10.}
\begin{align*}
 &t_1=(1,1,c,d)\\
&t_2=(\sqrt{c}/\sqrt{d},\sqrt{d}/\sqrt{c},\sqrt{c}\sqrt{d},\sqrt{c}\sqrt{d})\\
&t_3=(\sqrt{d}/\sqrt{c},\sqrt{d}/\sqrt{c},\sqrt{c}\sqrt{d},\sqrt{c}\sqrt{d})
\end{align*}
$W_t=(\mathbb{Z}/2\rtimes S_2)\rtimes\{1,\tau_1\}$ (case 1(l)).

\noindent\textbf{11.}
\begin{align*}
 &t_1=(-1,-1,c,d)\\
&t_2=(-\sqrt{c}/\sqrt{d},-\sqrt{d}/\sqrt{c},\sqrt{c}\sqrt{d},\sqrt{c}\sqrt{d})\\
&t_3=(\sqrt{d}/\sqrt{c},\sqrt{d}/\sqrt{c},-\sqrt{c}\sqrt{d},-\sqrt{c}\sqrt{d})
\end{align*}
$W_t=\mathbb{Z}/2\rtimes S_2$ (case 1(m)).

\noindent\textbf{12.}
\begin{align*}
 &t_1=(1,b,b,d)\\
&t_2=(b/\sqrt{d},\sqrt{d},\sqrt{d},b\sqrt{d})\\
&t_3=(\sqrt{d}/b,\sqrt{d},\sqrt{d},b\sqrt{d})
\end{align*}
$W_t=S_2\rtimes\{1,\tau_1\}$ and if $b=\sqrt{d}$ , we have $t_1=t_2=t_3$ and hence $W_t=S_2\rtimes S_3$ (case 1(n)).

\noindent\textbf{13.}
\begin{align*}
 &t_1=(1,b,c,d)\\
&t_2=(\sqrt{b}\sqrt{c}/\sqrt{d},\sqrt{b}\sqrt{d}/\sqrt{c},\sqrt{c}\sqrt{d}/\sqrt{b},\sqrt{b}\sqrt{c}\sqrt{d})\\
&t_3=(\sqrt{d}/\sqrt{b}\sqrt{c},\sqrt{b}\sqrt{d}/\sqrt{c},\sqrt{c}\sqrt{d}/\sqrt{b},\sqrt{b}\sqrt{c}\sqrt{d})
\end{align*}

$W_t=\{1,\tau_1\}$ and if $\sqrt{b}\sqrt{c}=\sqrt{d}$ then $t_1=t_2=t_3$ and $W_t=S_3$ (case 1(p)).

\noindent\textbf{14.}
\begin{align*}
 &t_1=(-1,b,b,d)\\
&t_2=(ib/\sqrt{d},i\sqrt{d},i\sqrt{d},-ib\sqrt{d})\\
&t_3=(i\sqrt{d}/b,-i\sqrt{d},-i\sqrt{d},ib\sqrt{d})
\end{align*}
$W_t=S_2$ (case 1(o)).

Next we consider $(t_1,t_2,t_3)$ such that none of the coordinates have $1$ as a parameter and all parameters of $t_i$ are distinct. Since
we are over an algebraically closed field $k$, Theorem \ref{strreg2}  holds in this case with the following modification:
\begin{theorem}
 If $t_i$ does not have $1$ as a parameter and all parameters in $t_i$ are distinct, $1\leq i\leq 3,$ 
 then $(t_1,t_2,t_3)$ is strongly regular in $F_4$. 
\end{theorem}
\begin{proof}
 Note that with the hypothesis on $t_i,$ $\mathfrak{C}^{t_i}=\{0\}$ for all $i$. For if not,
let $x(\neq 0)\in \mathfrak{C}^{t_i}$ for some $i$. Then
$t_i(x)=x\Rightarrow$ some parameter of $t_i$ is $1$ since $x$ is assumed to be non zero, a contradiction.
Also note that Lemma \ref{strreg} holds in this case too.
The rest of the proof is the same as that of Theorem \ref{strreg2}, with $\mathbb{R}$ replaced by $k$.
\end{proof}
We record the above discussion as
\begin{theorem}
 The genus number of a compact simply connected Lie group or a
 simply connected algebraic group over an algebraically closed field,
 of type $F_4$ is $17.$
\end{theorem}

\section{$G_2$}

\begin{definition}
 Let $\mathfrak{C}$ denote the octonion division algebra over $\mathbb{R}$. Then $Aut(\mathfrak{C})$ is the compact connected Lie group 
of type $G_2$.
\end{definition}
Conjugacy classes of centralizers in anisotropic forms of $G_2$ have been explicitly calculated in \textbf{[S]}. 
Here we count the number of such classes using 
a different technique.
Consider a maximal torus $T\subset G_2$. Then $T$ sits inside a copy of $SU(3)\subset G_2$.
If $K\subset \mathfrak{C}$ be a quadratic extension of $\mathbb{R},$ then $Aut(\mathfrak{C}/K)\cong SU(3),$
where $Aut(\mathfrak{C}/K)$ is the group of automorphisms of $\mathfrak{C}$ fixing $K$ point wise. The Weyl group of 
$G_2$ is $WG_2\cong WSU(3)\rtimes S_2,$ note that $S_2=Out(SU(3)).$ 
Let us consider the diagonal maximal torus $T$ in $SU(3)$ i.e. the one consisting of all 
diagonal matrices $t=(z_1,z_2,z_3)$, $z_i\in S^1$ and $z_1z_2z_3=1.$ The action of $WG_2$ on $T$ is given by
$$(\alpha,\beta)(z_1,z_2,z_3)=(\beta z_{\alpha^{-1}(1)},\beta z_{\alpha^{-1}(2)},\beta z_{\alpha^{-1}(3)}),$$
where $\alpha\in S_3$ , $\beta\in S_2$ and $\beta(z_i)=\overline{z_i}$ for $\beta\neq 1\in S_2.$
With this action, we now consider the various possibilities for an element $diag(z_1,z_2,z_3)\in SU(3)$ and calculate their stabilizers
in $WG_2.$

\noindent\textbf{(a)}  If $z_1\neq z_2\neq z_3,z_i$,then clearly $(WG_2)_t=\{1\}.$

\noindent\textbf{(b)}  If $z_1=z_2=z_3\in \mathbb{R}$ then $(WG_2)_t= S_3\rtimes S_2$.

\noindent\textbf{(c)}  If $z_1=z_2=z_3\in \mathbb{C-R}$ then $(WG_2)_t=S_3, $ since $Out(SU(3))$ acts non trivially.

\noindent\textbf{(d)}  If $z_1=z_2\neq z_3,z_i\in \mathbb{C-R}$  then $(WG_2)_t=S_2\subset WSU(3))$ as $Out(SU(3))$ acts non trivially.

\noindent\textbf{(e)}  If $z_1=z_2\neq z_3,z_i\in \mathbb{R}$ then $(WG_2)_t=S_2\rtimes S_2$ as $S_2$ leaves 
this element fixed and $S_2\subset
WSU(3)$ further acts trivially on it.

\noindent\textbf{(f)}  If $t=(1,exp(i\theta),exp(-i\theta))$ with $\theta \neq k\pi$ for any integer $k$, then 
$(WG_2)_t=\{(1,1),(\alpha,\beta)\}\equiv \mathbb{Z}/2$, where $\alpha\in S_3$ is the transposition $(2~3)$ and $\beta\in S_2$
is the transposition $(1~2).$

If we consider a connected algebraic group of type $G_2$ over an algebraically closed field $k,$
the semisimple genus number is the same. In this case, we work with the Zorn matrix model of split
 octonions and consider
$k\times k\subset \mathfrak{C}$ as the diagonal subalgebra. Then $Aut(\mathfrak{C})/(k\times k)\cong SL(3.)$ Consider the diagonal maximal
torus $T:=\{diag(a_1,a_2,a_3)\in SL(3)|a_1a_2a_3=1\}\subset SL(3),$ then $T$ is a maximal torus in $G_2.$ 
The Weyl group $G_2$ is $WG_2\cong WSL(3)\rtimes S_2\cong S_3\rtimes S_2. $ The action of
$WG_2$ on $T$ is given by $$(\alpha,\beta)(a_1,a_2,a_3)=(\beta a_{\alpha^{-1}(1)},\beta a_{\alpha^{-1}(2)},\beta a_{\alpha^{-1}(3)}),$$
where $\alpha\in S_3$ , $\beta\in S_2$ and $\beta(a_i)=1/{a_i}$ for $\beta\neq 1\in S_2$. 
The conjugacy classes of isotropy subgroups of $WG_2$ are as listed 
below:
(the arguments being same as the previous ones.)

\noindent\textbf{(a)} If $a_1\neq a_2\neq a_3$,$a_i\neq 1,-1$ and $a_i\neq 1/a_j$ for $i\neq j$, then $(WG_2)_t=\{1\}$

\noindent\textbf{(b)} If $a_i=1$ for all $i$, with , $W_t=(WG_2)$.

\noindent\textbf{(c)} If $a_i=\omega$ for all $i$,where $\omega$ is a cube root of unity other than $1$, $(WG_2)_t=S_3$.

\noindent\textbf{(d)} If $a_1=a_2\neq a_3$ with $a_1\neq 1,-1$, $(WG_2)_t=S_2$.

\noindent\textbf{(e)} If $a_1=a_2=1=-a_3$ then $(WG_2)_t=S_2\rtimes S_2$.

\noindent\textbf{(f)} If $a_1=1,a_2=1/a_3$ with $a_2\neq 1,-1$ then $(WG_2)_t=\{(1,1),(\alpha,\beta)\}\equiv \mathbb{Z}/2$,
 where $\alpha\in S_3$ is the transposition $(2~3)$ 
and $\beta\in S_2$ is the transposition $(1~2).$

The preceding discussion is recorded as,
\begin{theorem}
 The genus number of a compact simply connected Lie group 
or a simply connected algebraic group over an algebraically closed field,
  of type $G_2$ is $6.$
\end{theorem}
We now tabulate the results obtained so far:
\begin{center}
\begin{tabular}{|l|l|p{6 cm}|p{5cm}|}
\hline
\textbf{Group} &\textbf{Weyl group}&\textbf{Stabilizers} & \textbf{Genus Number}\\
\hline
$A_n$ & \small{$S_{n+1}$} &\small{$S_{n_1}...S_{n_k}$, where $n_1+...n_k=n+1$} & \small{$p(n+1)$}\\ \hline
$B_n$ &\small{ $(\mathbb{Z}/2)^n \rtimes S_n$}& 
\small{$(((\mathbb{Z}/2)^{i-1}\rtimes S_i)\times( (\mathbb{Z}/2)^{n_1-i}\rtimes S_{n_1-i}))\times S_{n_2}\times...\times S_{n_k}$ ,
where, $n_1+...+n_k=n$} &\small{ $\sum_{i=0}^n (i+1)p(n-i)$}\\ \hline
$C_n$ &$(\mathbb{Z}/2)^n \rtimes S_n$ 
&\small{$(((\mathbb{Z}/2)^i\rtimes S_i)\times ((\mathbb{Z}/2)^{n_1-i}\rtimes S_{n_1-i}))\times S_{n_2}\times...\times S_{n_k}$,
where $n_1+...n_k=n$} & \small{$\sum_{i=0}^n([i/2]+1)p(n-i)$}\\ \hline
$D_n$, $n$ odd &\small{$(\mathbb{Z}/2)^{n-1} \rtimes S_n$ }& 
\small{$(((\mathbb{Z}/2)^{i-1}\rtimes S_i)\times (\mathbb{Z}/2)^{n_1-i-1}\rtimes S_{n_1-i}))\times S_{n_2}\times...\times S_{n_k},$
where $n_1+...+n_k=n$}&
\small{$\sum_{i=0}^n ([i/2]+1)p(n-i)$}\\ \hline
$D_n$, $n=2k$ &\small{ $(\mathbb{Z}/2)^{n-1} \rtimes S_n$}&
\small{$(((\mathbb{Z}/2)^{i-1}\rtimes S_i)\times ((\mathbb{Z}/2)^{n_1-i-1}\rtimes S_{n_1-i})\times S_{n_2}\times...\times S_{n_l},$
where $n_1+...+n_l=n$ with at least one $n_i$ odd and $H(2k_1)\times S_{2k_2}\times...\times S_{2k_s},$
where $k_1+...+k_s=k$ and $H(2k_1)$ is a subgroup of order $(2k)!$ not conjugate to $S_{2k_1}$}&
\small{$\sum_{i=0}^n ([i/2]+1)p(n-i)) +p(k)$}\\ \hline
$F_4$&\small{$((\mathbb{Z}/2)^3\rtimes S_4)\rtimes S_3$}&As noted in Section 7&\small{17}\\ \hline
$G_2$&\small{$S_3\rtimes S_2$}& As noted in Section 8&\small{6}\\ \hline
\end{tabular}
\end{center}
\section{Computations for the Lie algebras}
If $G$ be a compact connected Lie group (or a connected reductive algebraic group over an algebraically closed field)
 with the Lie algebra denoted by $\mathfrak{g}$, the
orbit structure of the action of 
$Ad_G$ on $\mathfrak{g}$ can be neatly described in terms of the action of $WG$ on the Cartan subalgebra 
$\mathfrak{t}\subset \mathfrak{g}.$
 In this section we calculate the conjugacy classes of isotropy subgroups of $WG$ with respect to its action on $\mathfrak{t}.$ 
We begin with the following basic result ;
\begin{lemma}
 With respect to the action,
$Ad:G\longrightarrow Aut(\mathfrak{g})$ defined by $g\mapsto Ad_g,$ where $Ad_g(x)=gxg^{-1},$
(having embedded $G$ in a suitable $GL_n$) there is a bijection between the 
conjugacy classes of centralizers of semisimple elements in $\mathfrak{g}$ in $G$ and the conjugacy classes of centralizers of elements of 
a Cartan subalgebra  in $WG$.
\end{lemma}
\begin{proof}
 Consider the map $[G_x]\mapsto [WG_x],$ where $x\in \mathfrak{t}.$ To show this map a bijection we follow exactly
the same line of argument as in Theorems \ref{genus} and \ref{genus2}.
\end{proof}
For determining the stabilizers in the Weyl group we follow the same line of argument as in the case of groups in the previous sections.

\subsection{ $A_n$}
When $G$ is the Lie group $SU(n+1),$ the corresponding Lie algebra $\mathfrak{su}(n+1)$ is the set of all $(n+1)\times (n+1)$
 trace zero skew-hermitian matrices,while for $G=SL(n+1)$, $\mathfrak{g}$ consists of all trace zero $(n+1)\times (n+1)$ matrices. 
The Cartan subalgebra
in the above cases are given by:
$$\mathfrak{t}=\{(a_1i,...,a_{n+1}i)\in \mathbb{M}_n(\mathbb{C})|a_1+...+a_{n+1}=0\}\subset \mathfrak{su}(n+1)$$
and,
$$\mathfrak{t}=\{(a_1,...,a_{n+1})\in \mathbb{M}_n(k)|a_1+...+a_{n+1}=0\}\subset \mathfrak{sl}(n+1).$$
We have $WG=S_{n+1}$ and it acts on $\mathfrak{t}$ by permuting the entries in both cases. 
Hence by the argument followed in Section 3, we see that the 
number of conjugacy classes of isotropy subgroups is $p(n+1)$. The subgroups are of the form 
$S_{n_1}...S_{n_k}$  for a partition $(n_1,...,n_k)$ of $(n+1)$.
\subsection{$B_n$}
For the Lie algebra of type $B_n,$ the Cartan subalgebra
$\mathfrak{t}$ consists of all block diagonal matrices of the form $(A_1,...,A_n,0)$, where
$$A_i=\begin{bmatrix}
       0&a_i\\
-a_i&0
      \end{bmatrix}$$
is the $i-$th block with $a_i\in \mathbb{R}$.
And for $B_n$ over an algebraically closed field $k$ the Cartan subalgebra consists of all diagonal matrices of the form
$(a_1,...,a_n,-a_1,...,-a_n,0)$, where $a_i\in k.$ So in either situation we note that the elements of the Cartan subalgebra can be 
parametrized by the $n$-tuples $(a_1,...,a_n)$ with $a_i\in k$.
The Weyl group $W=(\mathbb{Z}/2)^n\rtimes S_n$ acts on $\mathfrak{t}$ by permuting the elements
, followed by a change of sign.

Let $(n_1,...,n_k) $ be a partition of $n$ such that $n_1 $ denotes the number of $0^,$s and $n_i$ for $i\neq 1$ denotes 
the number of equal parameters. For such an element the isotropy subgroup is 
$((\mathbb{Z}/2)^{n_1-1}\rtimes S_{n_1})\times S_{n_2}\times...\times S_{n_k}$
by an argument similar to one seen in $\S 4.$ Hence the number of isotropy classes is
$$\sum_{i=0}^n p(n-i).$$
\subsection{$C_n$}
The Cartan subalgebra $\mathfrak{t}$ consists of all diagonal matrices of the form
$(a_1,...,a_n,-a_1,...,-a_n)$ with $a_i\in k$.
The Weyl group being the same as that of $B_n$, we have the same number of isotropy classes in this case 
also,i.e $$ \sum_{i=0}^n p(n-i)$$.
\subsection{$D_n$}
Here the Cartan subalgebra is same as that of $B_n$ and the Weyl group $W=(\mathbb{Z}/2)^{n-1}\rtimes S_n$ acts on 
$\mathfrak{t}$ by permuting the parameters and changing the signs of an even number of them.

If $n$ is odd, then for a partition $(n_1,...,n_k)$ of $n$, where $n_i^,s$ are as in Section 9.2, the isotropy subgroup of the Weyl group
is $((\mathbb{Z}/2)^{n_1-1}\rtimes S_{n_1})\times S_{n_2}\times...\times S_{n_k}$ and hence the total number of isotropy classes is
$$\sum_{i=0}^n p(n-i).$$

However if $n=2k$,then if at least one zero occurs as one of the parameters of $t\in \mathfrak{t}$, then the isotropy subgroup is obtained
as above. But if no zero occurs i.e $n_1=0$, then for each partition of $n$ containing only even integers we have a isotropy subgroup not 
conjugate to any one of the above, as we have seen in the group case (see $\S 6.$). Thus the total number of isotropy classes for $n=2k$ is
$$\sum_{i=0}^n p(n-i) +p(k) .$$
\subsection{$G_2$}
In this case, we consider a subalgebra $\mathfrak{su}(3)$ (over reals) or $\mathfrak{sl}(3)$ (over an algebraically close field $k$)
inside $\mathfrak{g}_2$ and a Cartan subalgebra of $\mathfrak{g}_2$ embeds in one such subalgebra. 
Hence, each element of the Cartan subalgebra can be considered as all tuples $(a_1,a_2,a_3),$ $a_i\in k, $ such that $a_1+a_2+a_3=0.$ The 
Weyl group $WG_2\cong S_3\rtimes S_2$ (see section 9) acts on these tuples as,
$$(\alpha,\beta)(a_1,a_2,a_3)=(\beta a_{\alpha^{-1}(1)},\beta a_{\alpha^{-1}(2)},\beta a_{\alpha^{-1}(3)}),$$
where $\alpha\in S_3$ , $\beta\in S_2$ and $\beta(a_i)=-a_i$ for $\beta\neq 1\in S_2.$ Thus we have the following possibilities:

\noindent\textbf{(a)} If $t=(0,0,0)$ then clearly, $(WG_2)_t=WG_2$.

\noindent\textbf{(b)} If $t=(a,a,-2a)$ then $(WG_2)_t=S_2\subset WSL(3)$ since the other $S_2$ factor acts non trivially.

\noindent\textbf{(c)} If $t=(a,b,-a-b)$ with $a\neq b\neq -(a+b)$, then clearly, $(WG_2)_t=\{1\}$.

\noindent\textbf{(d)} If $t=(0,a,-a)$ with $a\neq 0$ then $(WG_2)_t=\{(1,1),(\alpha,\beta)\}\cong \mathbb{Z}/2,$
 where $\alpha=(2~3)\in S_3$ 
and $\beta=(1~2)\in S_2$.
\subsection{$F_4$}
Here we will use the notations used in Section 7.
We work with the basis
of $\mathfrak{C}$ i.e $\{v_1,...,v_8\}$
as in Section 7.
We reorder this basis as $e_1=v_1,e_2=v_2,e_3=v_3,e_4=v_4,e_5=v_8,e_6=v_7,e_7=v_6,e_8=v_5$ so  that with respect to the new ordered basis
$\{e_1,...,e_8\},$
 the matrix of the bilinear form associated with the norm $N$ (see section 8) of $\mathfrak{C}$ becomes
$$\begin{bmatrix}
   0&I\\
I&0
  \end{bmatrix}.$$
 Also, the Cartan subalgebra of $\mathfrak{so}(N)$ is in the diagonal form with  respect to the above bilinear form,
i.e. $\mathfrak{t}\subset \mathfrak{so}(N)$ will consist of all diagonal matrices of the form
$(a_1,...,a_4,-a_1,...,-a_4),$ $a_i\in k.$ Henceforth we shall parametrize this diagonal matrix as $(a_1,a_2,a_3,a_4),$ $a_i\in k.$
The Cartan subalgebra of $\mathfrak{f}_4$ is contained in a copy of the Lie algebra of $Spin(N),$ i.e.
$\mathfrak{spin}(N)\cong \mathbf{L}(RT(\mathfrak{C})),$ where
$\mathbf{L}(RT(\mathfrak{C}))=\{(t_1,t_2,t_3)\in \mathfrak{so}(8)^3|t_1(xy)=t_2(x)y+xt_3(y),x,y\in\mathfrak{C}\}.$
It is known that $\mathfrak{so}(N)$ is generated as a vector space by $t_{a,b}$, $a,b\in \mathfrak{C}$; $t_{a,b}$
is defined as $t_{a,b}(x)= \langle x,a\rangle b-\langle x,b \rangle a$ for $x\in \mathfrak{C}$
where $\langle ,\rangle$ is the bilinear form of the norm $N$ (\textbf{[SV]}, Chapter 3).

If $t_1=t_{a,b},$ then $t_2=1/2(l_bl_{\overline{a}}-l_al_{\overline{b}})$ and 
$t_3=1/2(r_br_{\overline{a}}-r_ar_{\overline{b}})$ satisfy the property,
\begin{align}{\label{c}}t_1(xy)=t_2(x)y+xt_3(y).\end{align}
Also note that if $(t_1,t_2,t_3)$ and $(s_1,s_2,s_3)$ are related triples (in the Lie algebra sense) then so is
$(t_1+s_1,t_2+s_2,t_3+s_3)$. With this, we can now carry out the computation.

Let $t_1=(a_1,a_2,a_3,a_4)$. Then by a direct computation using the multiplication table for the 
basis $\{v_i\}$ in Section 8  and (\ref{c}), 
one can show that
  $t_1=\sum_{i=1}^4 t_{x_i,y_i}$, where 
$x_i,y_i $ are given by $x_i=a_i(e_i+e_{4_i})$ and $y_i=(e_i-e_{4+i})/2$. Using this, the above formulas for $t_2$ and $t_3$ and the
 multiplication table for the $v_i^,s$ (see section 8), we get,
\begin{align*}
 &t_1=(a_1,a_2,a_3,a_4)\\
&t_2=((a_1+a_2+a_3-a_4)/2,(a_1+a_2-a_3+a_4)/2,\\&(a_1-a_2+a_3+a_4)/2,(-a_1+a_2+a_3+a_4)/2)\\
&t_3=((a_1-a_2-a_3+a_4)/2,(-a_1+a_2-a_3+a_4)/2,\\&(-a_1-a_2+a_3+a_4)/2,(a_1+a_2+a_3+a_4)/2)
\end{align*}
Also note that if $t=(a_1,a_2,a_3,a_4)$ then $\hat{t}=(-a_1,a_2,a_3,a_4).$ This is evident from the fact that $\bar{e_1}=e_5$ and
$\bar{e_i}=-e_i$ whenever $i\neq 1,5$ and the definition of $\hat{t}$ i.e. $\hat{t}(x)=\overline{t(\bar{x})},$ $x\in \mathfrak{C}.$
 We refer to the above set of equations by (A). Recall that the Weyl group of $F_4$ is $W\cong WSpin(N)\rtimes S_3\cong 
((\mathbb{Z}/2)^3\rtimes S_4)\rtimes S_3$ and the action of $W$ on $\mathbf{L}RT(\mathfrak{C})$ is given by (\ref{a}).

We now calculate the stabilizers of elements of $\mathbf{L}(RT(\mathfrak{C}))$ in $W,$ the arguments being similar to those for the group
$F_4.$

\noindent\textbf{(1)} By (A),
\begin{align*}
 t_1=t_2=t_3=0
\end{align*}
 Then clearly $W_t=WF_4.$

\noindent\textbf{(2)} If

$t_1=(0,0,0,a_4),$ then by (A),
\begin{align*}
&t_2=(-a_4/2,a_4/2,a_4/2,a_4/2)\\
&t_3=(a_4/2,a_4/2,a_4/2,a_4/2) 
\end{align*}
Here we observe that only $\tau_1$ fixes $t$ since $t_2,t_3$ do not have $0$ as a parameter, no other element from $S_3=Out(Spin(N))$
can contribute to the isotropy (see \ref{a}). Thus
$W_t=((\mathbb{Z}/2)^2\rtimes S_3)\rtimes \{1,\tau_1\}.$

\noindent\textbf{(3)} If $t_1=(0,0,a_3,a_3),$ then by $(A),$
\begin{align*}
 t_1=t_2=t_3=(0,0,a_3,a_3)
\end{align*}
Therefore, $\hat{t_1}=\hat{t_2}=\hat{t_3}.$ Hence all of $S_3=Out(Spin(N))$ fixes $t$ (see \ref{a}). Therefore,
$W_t=((\mathbb{Z}/2\rtimes S_2)\times S_2)\rtimes S_3$ 

\noindent\textbf{(4)} If
 $t_1=(0,0,a_3,a_4),$ then by (A)
\begin{align*}
&t_2=\hat{t_3}=((a_3-a_4)/2,(a_4-a_3)/2,(a_3+a_4)/2,(a_3+a_4)/2).
\end{align*} We have,
$W_t=(\mathbb{Z}/2\rtimes S_2)\rtimes \{1,\tau_1\},$ because apart from $\tau_1$ any other element of $S_3$ sends $t_2$ or $t_3$
to the first position (see \ref{a}) and hence they cannot fix $t$.

\noindent\textbf{(5)} If $ t_1=(0,a_2,a_2,a_2),$ then by (A),
\begin{align*}
&t_2=\hat{t_3}=(a_2/2,a_2/2,a_2/2,a_2/2).
\end{align*}
 Since $t_2=\hat{t_3},$ only $\tau_1\in S_3$ appears in the isotropy subgroup (see \ref{a}). Therefore, $W_t=S_3\rtimes\{1,\tau_1\}.$

\noindent\textbf{(6)} If
$t_1=(0,a_2,a_2,a_4),$ then by (A),
\begin{align*}
&t_2=\hat{t_3}=((2a_2-a_4)/2,a_4/2,a_4/2,(2a_2+a_4)/2).
\end{align*}
We have, $W_t=S_2\rtimes \{1,\tau_1\}$ if $2a_2\neq a_4$ and if $a_4=2a_2$ then $t_1=t_2=t_3$ and $S_3$ will clearly fixes
 $t$ (see \ref{a}). Hence $W_t=S_2\rtimes S_3.$

\noindent\textbf{(7)} If
$t_1=(0,a_2,a_3,a_4),$ then by (A),
\begin{align*}
&t_2=((a_2+a_3-a_4)/2,(a_2-a_3+a_4)/2,(-a_2+a_3+a_4)/2,(a_2+a_3+a_4)/2)\\
&t_3=\hat{t_2}
\end{align*}
If $t_2,t_3$ does not contain $0$ as a parameter, then $W_t=\{1,\tau_1\}\subset S_3$ since any other element of $S_3$ removes $t_1$ from 
the first position of the related triple by \ref{a}.
 Otherwise, let $a_2+a_3-a_4=0$, then by (A), $t_1=t_2=t_3$ and therefore, $S_3$ stabilizes $t.$ In this case, $W_t=
\{1\}\rtimes S_3$. For the other three possibilities the related triple can be made Weyl group equivalent to the latter by a suitable 
permutation of $a_2,a_3,a_4.$

\noindent\textbf{(8)} If 
$t_1=(a_1,a_1,a_3,a_4),$  then by (A),
\begin{align*}
&t_2=((2a_1+a_3-a_4)/2,(2a_1-a_3+a_4)/2,(a_3+a_4)/2,(a_3+a_4)/2)\\
&t_3=((-a_3+a_4)/2,(-a_3+a_4)/2,(-2a_1+a_3+a_4)/2,(2a_1+a_3+a_4)/2)
\end{align*} We have
$W_t=S_2\subset WSpin(N),$ since every element of $S_3$ other than $1,$ acts non trivially on $t$ (see \ref{a}).

\noindent\textbf{(9)} If
$t_1=(a_1,a_1,a_1,a_4),$ then by (A)
\begin{align*}
&t_2=((3a_1-a_4)/2,(a_1+a_4)/2,(a_1+a_4)/2,(a_1+a_4)/2)\\
&t_3=((-a_1+a_4)/2,(-a_1+a_4)/2,(-a_1+a_4)/2,(3a_1+a_4)/2).
\end{align*} We have,
$W_t=S_3\subset Spin(N)$ because $a_1\neq a_4$ and hence only elements from $WSpin(N)$ fixes $t$ (see \ref{a}).

\noindent\textbf{(10)} If
$t_1=(a_1,a_2,a_3,a_4),$ then by (A),
\begin{align*}
&t_2=((a_1+a_2+a_3-a_4)/2,(a_1+a_2-a_3+a_4)/2,\\&(a_1-a_2+a_3+a_4)/2,(-a_1+a_2+a_3+a_4)/2)\\
&t_3=((a_1-a_2-a_3+a_4)/2,(-a_1+a_2-a_3+a_4)/2,\\&(-a_1-a_2+a_3+a_4)/2,(a_1+a_2+a_3+a_4)/2)
\end{align*}
Here, the isotropy subgroup is trivial if none of the $t_i^,$s contain $0$ as parameter, because in that case all non trivial elements
 of $S_3$
 act non trivially on
$(t_1,t_2,t_3)$ (see \ref{a}).

Hence there are 12 conjugacy classes of isotropy subgroups in the Weyl group.

We conclude this section by collecting the results for Lie algebras in the following table:
\begin{center}
 \begin{tabular}{|l|p{3cm}|p{6 cm}|p{5cm}|}
\hline
Lie algebra&Weyl group&Stabilizers&number of orbit types \\ \hline
$A_n$&$S_{n+1}$&$S_{n_1}...S_{n_k}$  for a partition $n_1,...,n_k$ of $n+1$&$p(n+1)$\\ \hline
$B_n$&$(\mathbb{Z}/2)^n\rtimes S_n$&$((\mathbb{Z}/2)^{n_1-1}\rtimes S_{n_1})\times S_{n_2}\times...\times S_{n_k}$&
$\sum_{i=0}^n\quad p(n-i)$\\ \hline
$C_n$&$(\mathbb{Z}/2)^n\rtimes S_n$&$((\mathbb{Z}/2)^{n_1}\rtimes S_{n_1})\times S_{n_2}\times...\times S_{n_k}$&
$\sum_{i=0}^n\quad p(n-i)$\\ \hline
$D_n$for $n$odd&$(\mathbb{Z}/2)^{n-1}\rtimes S_n$&$((\mathbb{Z}/2)^{n_1-1}\rtimes S_{n_1})\times S_{n_2}\times...\times S_{n_k}$&
$\sum_{i=0}^n\quad p(n-i)$\\ \hline
$D_n$ for $n=2k$&$(\mathbb{Z}/2)^{n-1}\rtimes S_n$&$((\mathbb{Z}/2)^{n_1-1}\rtimes S_{n_1})\times S_{n_2}\times...\times S_{n_k}$,
and for each partition $k_1,..,k_s$ of $k,$ $H_{2k_1}.S_{2k_2}...S_{2k_s},$ where $H_{2k_1}$ is a
 subgroup of order $(2k_1)!$ not conjugate to $S_{2k_1}$.&$\sum_{i=0}^n\quad p(n-i)+p(k) $\\ \hline
$G_2$&$S_3\rtimes S_2$&refer to the discussion above&$4$\\ \hline
$F_4$&$((\mathbb{Z}/2)^3\rtimes S_4)\rtimes S_3$&refer to the discussion above&$12$\\ \hline

\end{tabular}

\end{center}


\textbf{Acknowledgements:}
I thank Maneesh Thakur for suggesting me the problem and all the discussions I had with him and Anupam Singh from IISER,
Pune, who had lent me some of his valuable time in discussing this work. 
I also thank Dipendra Prasad from TIFR, Bombay, for his encouragement. I also thank Professor Donna Testerman from EPFL, Lausanne, for some
extremely crucial  
comments and suggestions on the manuscript.

\end{document}